\documentclass[reqno]{amsart}
 \usepackage{verbatim}
 \usepackage[mathscr]{eucal}
 \usepackage{amssymb,bbm}
 \usepackage{thmtools} 
 \usepackage{mathtools} 
\usepackage{url}
\usepackage{pgf}
\usepackage{tikz}
\usetikzlibrary{positioning}
\usetikzlibrary{arrows,shapes,calc,backgrounds,fit}
\tikzset{%
 unshaded/.style={draw, shape=circle, fill=white, inner sep=1.5pt},
 shaded/.style={draw, shape=circle, fill=black, inner sep=1.5pt},
 invisible/.style={shape=circle, inner sep=1.5pt},
 label/.style={shape=rectangle, inner xsep=5pt, inner ysep=7pt},
 auto,
 curvy/.style={->, shorten >=3pt, shorten <=3pt, >=latex, looseness=1, bend angle=30},
 straight/.style={->, shorten >=3pt,shorten <=3pt, >=latex},
 loopy/.style={->, shorten >=3pt, shorten <=3pt, >=latex, min distance=15pt},
 order/.style={thin},
 bigloop/.style={-, min distance=65pt}}
 \pgfdeclarelayer{background}
 \pgfdeclarelayer{foreground}
 \pgfsetlayers{background,main,foreground}

\usepackage{amsmath} 

\newtheorem{theorem}{Theorem}[section]
\newtheorem{lemma}[theorem]{Lemma}
\newtheorem{proposition}[theorem]{Proposition}
\newtheorem{corollary}[theorem]{Corollary}
\newtheorem{TopSwap}[theorem]{TopSwap Theorem for Structures}

 \theoremstyle{definition}
\newtheorem{definition}[theorem]{Definition}
\newtheorem{remark}[theorem]{Remark}
\newtheorem{example}[theorem]{Example}

\declaretheoremstyle[
spaceabove=1ex, spacebelow=1ex,
headfont=\normalfont\bfseries,
notefont=\bfseries, notebraces={}{},
bodyfont=\normalfont,
postheadspace=0.5em,
name={\ignorespaces},
numbered=no,
headpunct=.]
{mystyle}
\declaretheorem[style=mystyle]{named}

\newenvironment{newlist}
 {\begin{list}{}{\setlength{\labelsep}{0.25cm}
 \setlength{\labelwidth}{0.65cm}
 \setlength{\leftmargin}{0.9cm}}}  
 {\end{list}}

\newenvironment{newitemize}
 {\begin{list}{$\bullet$}{\setlength{\labelsep}{0.25cm}
 \setlength{\labelwidth}{0.25cm}
 \setlength{\leftmargin}{0.45cm}}}
 {\end{list}}

 \newcommand{\defn}[1]{{\emph{#1}}}

\newcommand{\cat}[1]{\boldsymbol{\mathscr{#1}}}
\newcommand{\CA}{{\cat A}}
\newcommand{\CB}{{\cat B}}
\newcommand{\CC}{\cat C}
\newcommand{\CCD}{\cat D}
\newcommand{\CK}{\cat K}
\newcommand{\CS}{\cat S}

\newcommand{\CP}{\cat P}
 
\newcommand{\CM}{{\cat M}}

\newcommand{\CO}{\cat O}

\newcommand{\D}[1]{\mathrm D(#1)}
\newcommand{\E}[1]{\mathrm E(#1)}
\newcommand{\F}[1]{\mathrm F(#1)}
\newcommand{\G}[1]{\mathrm G(#1)}
\newcommand{\ED}[1]{\mathrm {ED}(#1)}
\newcommand{\DE}[1]{\mathrm {DE}(#1)}
\newcommand{\GF}[1]{\mathrm {GF}(#1)}
\newcommand{\FG}[1]{\mathrm {FG}(#1)}

\newcommand{\CX}{\cat X}
\newcommand{\CZ}{\cat Z}

\DeclareMathOperator{\Filt}{Filt}

\newcommand{\CPb}{\cat P_{\kern -2pt\scriptscriptstyle 01}}

\newcommand{\sub}[1]{_{_{\kern-.9pt{\scriptstyle #1}}}}
\newcommand{\medsub}[2]{#1\lower0.6ex\hbox{$\scriptstyle{#2}$}}
\newcommand{\eA}[1]{\medsub e {\kern-0.75pt\A\kern-0.75pt}(#1)}
\newcommand{\esub}[1]{\medsub e {\kern-0.75pt #1 \kern-0.75pt}}
\newcommand{\epsub}[1]{\medsub \varepsilon {\kern-1.25pt #1}}
\newcommand{\esubA}{\medsub e {\kern-0.75pt\A\kern-0.75pt}}
\newcommand{\twiddle}[1]{{\smash{\underset{\raise.4ex\hbox{$\smash\sim$}}
                       {\mathbf{#1}}}\vphantom{\underline{\mathbf{#1}}}}}
\newcommand{\stwiddle}[1]{{\smash{\underset{\raise.3ex\hbox{$\scriptstyle\smash\sim$}}{\mathbf{#1}}}\vphantom{\underline{\underline{\mathbf{#1}}}}}}

\newcommand{\MT}{\mathbf M^\Tp}

\newcommand{\twoBf}{\mathbf 2}
\newcommand{\twoT}{\mathbbm 2}

\newcommand{\PP}{\mathbf Y}

\newcommand{\T}{\mathscr{T}}

\newcommand{\A}{\mathbf A}
\newcommand{\B}{\mathbf B}

\newcommand{\Lalg}{\mathbf L}
\newcommand{\M}{\mathbf M}

\newcommand{\X}{\mathbf X}
\newcommand{\Y}{\mathbf Y}
\newcommand{\Z}{\mathbf Z}

\newcommand{\Ys}{\Y\kern -2pt _s}

\newcommand{\Tp}{\mathscr{T}}
\renewcommand{\leq}{\leqslant}

\newcommand{\dotbigcupdisp}{{\Large\textstyle\overset{\raise.65ex\hbox{$\smash\cdot$}}{\smash\bigcup}}}

\newcommand{\dotbigcup}{\overset{\raise.6ex\hbox{$\smash\cdot$}}{\smash\bigcup}}

\newcommand{\n}{n}

 \DeclareMathOperator{\Endo}{End}

 \DeclareMathOperator{\eq}{eq}
 \DeclareMathOperator{\id}{id}
 
 \DeclareMathOperator{\pro}{pro\!}

 \newcommand{\ISP}{\mathsf{ISP}}

 \newcommand{\ISOP}{\mathsf{IS}^0\mathsf{P}}
 \newcommand{\ISOPp}{\mathsf{IS}^0\mathsf{P}^+}
 
 \newcommand{\HSP}{\mathsf{HSP}}
 
 \newcommand{\IScP}{{\mathsf{IS} _{\mathrm{c}}
 \mathsf{P}^+}}
 \newcommand{\IScPnp}{{\mathsf{IS} _{\mathrm{c}}
 \mathsf{P}}}
 \newcommand{\ISOcP}{{\mathsf{IS}^0 _{\mathrm{c}}
 \mathsf{P}^+}}
\newcommand{\ISOcPnp}{{\mathsf{IS}^0 _{\mathrm{c}}
 \mathsf{P}}}

\newcommand{\rest}[1]{{\upharpoonright}_{#1}}
\newcommand{\bigand}{\mathop{\bigwedge\kern -8.5truept \bigwedge}}
\newcommand{\Bigand}{\mathop{\bigwedge\kern -10truept \bigwedge}}
\newcommand{\littleand}{\mathbin{\wedge\kern -8truept \wedge}}
\newcommand{\bigor}{\mathop{\bigvee\kern -8.5truept \bigvee}}
\newcommand{\Bigor}{\mathop{\bigvee\kern -10truept \bigvee}}
\newcommand{\littleor}{\mathbin{\vee\kern -8truept \vee}}

\newcommand{\lsem}{[\kern-1.75pt[}
\newcommand{\rsem}{]\kern-1.75pt]}

\newcommand{\Nach}[1]{\beta_\leq(#1)} 
\newcommand{\upsh}[1]{{\upshape#1}}
\newcommand{\powerset}[1]{{\raise.5ex\hbox{\Large$\wp$}(#1)}}

\begin{document}

\title[Compactifications of structures]
{Bohr compactifications of\\ algebras and structures}

\author{B.\,A. Davey} 
\address{Department of Mathematics and Statistics, La Trobe University, Victoria 3086,
           Australia}
\email{B.Davey@latrobe.edu.au}
\author{M. Haviar}
\address{Faculty of Natural Sciences, Matej Bel University, Tajovsk\'eho 40,
           974 01 Bansk\'{a} Bystrica, Slovak Republic}
\email{miroslav.haviar@umb.sk}

\author{H.\,A. Priestley}
\address{Mathematical Institute, University of Oxford, Radcliffe Observatory Quarter,
           Oxford OX2 6GG, UK}   
\email{hap@maths.ox.ac.uk}
\thanks{The first author wishes to thank the Research Institute of M.~Bel University in Bansk\'{a} Bystrica for its hospitality while working on this paper. The second author acknowledges support from Slovak grants APVV-0223-10, Mobility-ITMS 26110230082 and VEGA 1/0212/13.}

\begin{abstract}
This paper provides a unifying framework for a range of categorical constructions characterised by universal mapping properties, within the realm of compactifications of discrete structures. Some classic examples fit within  this broad picture: the Bohr compactification of an abelian group via Pontryagin duality, the zero-dimensional Bohr compactification of a semilattice, and the Nachbin order-compactification of an ordered set.

The notion of a natural extension functor is extended to suitable categories of structures  and such a functor is shown to  yield a reflection into an associated category of topological structures.  Our principal results address reconciliation of the natural extension with the Bohr compactification  or its zero-dimensional variant. In certain cases  the natural extension functor and a Bohr compactification functor are the same;  in others the functors have different codomains but may agree on all objects. Coincidence in the stronger sense occurs in the zero-dimensional setting precisely when the domain is a category of structures whose associated topological prevariety is standard. 
It occurs, in the weaker sense only, for the class  of ordered sets and,  as we show, also for infinitely many classes of ordered structures. 

Coincidence results aid understanding of Bohr-type compactifications, which are defined abstractly. Ideas from natural duality theory lead to an explicit description  of the natural extension which is particularly amenable for any   prevariety of algebras with a finite, dualisable, generator. Examples of such classes---often varieties---are plentiful and varied, and in many cases the associated topological prevariety is standard.  
\end{abstract}


\maketitle

\section{Introduction}  \label{sec:intro}

Our purpose
is to bring within a common framework
a range  of apparently rather disparate universal constructions.
In all cases the objects constructed are topological structures and the construction
is performed by applying the left adjoint to a functor which forgets the topology.
Constructions  of this type arise
widely in algebra and in topology, under  various guises. Specific examples include
\begin{newitemize}  
\item  
the Bohr compactification of an abelian group \cite{Ho64};

\item
the Bohr compactification of a unital meet semilattice \cite{HMS74};

\item
the Stone--\v {C}ech compactification of a set;

\item
 the Nachbin order-compactification of an ordered set \cite{N76,BeMo}.
\end{newitemize}
Here the categories on which the left adjoint functors act have as objects
a suitable class  either of algebras or of  relational structures.

The Bohr compactification is best known, and first  received  attention,
in the context of topological abelian groups.
Somewhat later, the ideas were extended to  semigroups, semilattices and rings
by Holm \cite{Ho64}, suggesting
that a theory of Bohr compactifications could be developed
for  algebraic structures more widely.
This was taken forward by Hart and Kunen \cite{HK99}.
However
they work with an algebraic first-order language, so that the discrete structures
of their title are less general than those we shall consider. 
 (It is immaterial whether
one chooses overtly  to include the discrete topology or, as we shall do, 
suppress it.)
The Bohr construction  comes in two distinct flavours, depending on whether
 one seeks a reflection into
 a category of topological structures  which
has objects which
carry a  compact Hausdorff topology
or one in which the objects are compact and zero-dimensional.
These functors are customarily denoted, respectively,
 by~$b$ and~$b_0$.

Bohr compactifications may be considered alongside other
generic  constructions one may perform
on suitable classes:
\begin{newitemize}
\item  
Bohr compactifications  and zero-dimensional Bohr compactifications,  of  algebraic structures,
as studied in \cite{HK99};
\item

 the natural extension of an algebra in any internally residually finite prevariety, 
abbreviated IRF-prevariety,
that is, a class of the form $\ISP(\CM)$, where $\CM$ is a set of finite algebras of common
signature \cite{nisp};
\item
  the profinite completion of an algebra in a residually finite
variety or more generally an IRF-prevariety, in general and in particular cases
(see
\cite{Num57, 
BGMM,
DHP07,DP12} and references therein);

\item

 the canonical extension of an algebra in a finitely generated variety of lattice-based algebras (see \cite{DP12} and references therein).
\end{newitemize}
In each of these cases we start from a category~$\CA$ of algebras
and consider a category~$\CB$ of topological structures in which each object
has a topology-free reduct in~$\CA$  and we have a functor $\mathrm F \colon \CA \to \CB$ which is left adjoint to the functor from $\CB$ to $\CA$ which
 forgets the topology.
The constructions differ in their scope
(that is, in the conditions on the  domain category~$\CA$ on which they operate) 
and  in the manner in which they are
customarily formulated.  Where the same category~$\CA$ supports
more than one of the constructions, the codomain category~$\CB$  may vary.
  In some instances existence is initially   established abstractly;
in others, and for the natural extension in particular,  a concrete description is presented at the outset, or can be derived.
A recurrent theme however is that the constructions can be characterised by  an appropriate universal mapping property.

Our presentation of an overarching framework for compactifications of Bohr type relies on widening the scope of the natural extension construction.
We consider prevarieties  of the form $\CA= \ISP(\CM)$, where $\CM$ is a set of structures
and each $\M\in \CM$ has an associated
compact Hausdorff topology $\Tp$ that is compatible with the structure;
we call such a class a
\emph{compactly-topologisable prevariety}, or CT-prevariety for short.
We then define the associated topological prevariety $\CA_\Tp \coloneqq \IScPnp(\CM_\Tp)$, where $\CM_\Tp =\{\,\M_\Tp\mid \M\in \CM\,\}$ and $\M_\Tp$ denotes $\M$ endowed with the topology~$\Tp$. 
We make both $\CA$ and $\CA_\Tp$ 
into 
categories in the obvious way  (see Sections~\ref{natBohr}
and ~\ref{natext} for more details).
Frequently  we shall present theoretical results only for the case $|\CM| = 1$;
this simplifies the presentation and covers the specific classes we target
in this paper.

A natural
extension functor $n_{\CA}$ exists on any CT-prevariety $\CA$ and so is available in particular
on all four of the categories in our initial list:  abelian groups, unital 
meet semilattices, sets, and ordered sets.
We demonstrate in Section~\ref{natext} that many of the good features of the natural  extension revealed in \cite{nisp} extend to the wider setting:  most notably we have a well-defined functor which is  a reflection and so acts as the left adjoint to the forgetful functor 
from $\CA_\Tp$ into $\CA$.
(We note, by contrast,
that the profinite completion construction
cannot be expected to extend beyond the setting of IRF-prevarieties of algebras.)
The natural extension construction has an important virtue.
The formalism of traditional
natural duality theory, as presented in the text of Clark and Davey \cite{NDftWA},
 enables an explicit description to be given in general of
$n_{\CA}(\A)$, for $\A$ in an IRF-prevariety  $\CA$ of algebras, and
 in a more refined and amenable form, for classes which
admit a natural duality~\cite{nisp}.
Drawing similarly on duality  theory ideas
that
extend to IRF-prevarieties of  structures  \cite{D06} and to certain
CT-prevarieties of structures
\cite{pig},  we are able explicitly to describe natural extensions in the
wider setting, though there are impediments:
structures must not contain
partial operations and,  for  prevarieties with infinite generators,
the results we obtain are less complete  than those for IRF-prevarieties.

The Bohr compactification functor $b$ on a prevariety $\CA$ of structures 
(as presented in Section~\ref{natBohr})
maps $\CA$ into the category $\CA^{\mathrm{ct}}$ of compact topological structures with non-topological reduct in~$\CA$ and is left adjoint to the natural forgetful functor. The zero-dimensional Bohr compactification functor $b_0$ is defined similarly, with $\CA^{\mathrm{ct}}$ replaced by the category $\CA^{\mathrm{Bt}}$ of Boolean topological structures with reduct in~$\CA$.
Thus a Bohr compactification  has an abstract characterisation,  and  so is
 hard to describe explicitly.
  It is therefore advantageous to know when it coincides with the more readily accessible natural extension.  In Section~\ref{sec:BviaN}
we  elucidate
the relationship between the natural extension functor $n_{\CA}$
on a class~$\CA$ of structures and the functor $b$ and,
 when the objects of $\CA_\Tp$ are zero-dimensional, 
the functor~$b_0$; 
see Proposition~\ref{prop:coincide}. The situation
is illustrated in Figure~\ref{fig:new}; each of the functors is left adjoint to the corresponding forgetful functor.

\begin{figure}[ht]
\begin{tikzpicture}
 [node distance=1.5cm,
 auto,
 text depth=0.25ex,
 move down/.style= {transform canvas={yshift=-1pt}}]
\node (A) {$\CA$};
\node (ABt) [above=of A] {$\CA^{\mathrm{Bt}}$};
\node (AT) [left=of ABt] {$\CA_\Tp$};
\node (Act) [right=of ABt] {$\CA^{\mathrm{ct}}$};
\draw[-latex] (A) to node {$n_{\CA}$} (AT);
\draw[-latex] (A) to node [swap] {$b_0$} (ABt);
\draw[-latex] (A) to node [swap] {$b$} (Act);
\draw[right hook-latex,move down] (AT) to node {} (ABt);
\draw[right hook-latex,move down] (ABt) to node {} (Act);
\end{tikzpicture}
\caption{The functors $n_{\CA}$, $b_0$ and $b$ 
(in the case that $\CA_\Tp \subseteq \CA^{\text{Bt}}$)
\label{fig:new}}
\end{figure}
The coincidence results we present 
are  a core constituent of the paper.  They
are of two types. \emph{Strong coincidence} occurs when the functors under consideration can be shown to have the same codomain, from which it follows that the functors are identical, having the same domain, codomain and values.
\emph{Weak coincidence} arises when the codomain
categories are different but the image of the functor into the larger 
of the categories lies in the smaller one: for example, if $b_0$ maps $\CA$ into $\CA_\Tp$, then $b_0(\A) = n_\CA(\A)$, for all $\A\in \CA$, and hence $b_0$ and $n_\CA$ coincide except for their codomains. 

Given an  IRF-prevariety~$\CA$,  strong coincidence  of
$b_0$ and $n_{\CA} $ occurs exactly when the associated 
topological prevariety $\CA_\Tp$ is standard.  The notion of a \emph{standard
topological prevariety} has received considerable attention in its own right
\cite{CDHPT,CDFJ,CDJP,CDMM,DT,J08}.  
This  literature enables us to present an extensive list of 
IRF-prevarieties of algebras (all of which are in fact  varieties) 
for which the zero-dimensional
Bohr compactification coincides
strongly
 with the natural extension 
 and can thereby be explicitly described with the aid of known dualities
(see Theorems~\ref{thm:pairedadjunctions} and~\ref{thm:b0-sumup}).
We also consider briefly, with examples, 
new notions of standardness appropriate to 
a CT-prevariety which has an infinite generator. 

We can  draw on  two famous examples from the literature to highlight 
instances of non-coincidence which arise in different ways.  The IRF-prevariety
$\CS$ of unital meet semilattices is standard, so that
strong coincidence occurs for $n_{\CS}$ and $b_0$.  However weak coincidence
of $b_0$ and $b$  fails; see Theorem~\ref{thm:badsemilat}.  Now consider
the IRF-prevariety~$\CP$ of ordered sets.  Here 
standardness fails (see Example~\ref{ex:Stralka})
and strong coincidences 
are ruled out.  However weak coincidence of $b$ (the Nachbin order compactification functor)  and $n_{\CP}$ does occur, and so,
even though they have three different codomains, the functors $b$, $b_0$ and $n_{\CP}$ take the same values on~$\CA$: so $b(\PP) = b_0(\PP) = n_{\CA}(\PP)$, for each ordered set~$\PP$;   
see 
Proposition~\ref{NachisNat}.
Building on this example,  Theorem~\ref{thm:pigcoincide}
supplies  a countably infinite family of IRF-prevarieties~$\CX$ of ordered structures exhibiting the same  behaviours as does~$\CP$.
Underpinning our discovery of these prevarieties is the method of
topology-swapping, originating in \cite{DHP12} and applied in Section~\ref{sec:egViaDuality} to the description of natural extensions in linked pairs of categories; see Corollary~\ref{descr_of_compatifications}.

We should issue a reassurance that readers of this paper are not assumed
to have
a working knowledge of
natural
duality theory.
  As we have indicated, our key tools
for identifying zero-dimensional Bohr compactifications are the natural extension
construction and the notion of a standard
topological prevariety.
These tools are an adjunct to, rather than a part of, duality theory and our presentation of the theory we require is self-contained.
We do however refer to the literature for results concerning dualisability
or
otherwise of particular classes of structures, and for
details of particular dualities, where these exist.

We conclude this introduction by stressing that our objective
is to analyse  in a uniform manner compactifications in a range of
specific categories.
Our focus is 
very different from that of the treatment of universal constructions
within an abstract categorical framework, as
presented  in such sources as \cite{Jo80} or \cite{Jo90}.
Our account, by contrast, does have
some affinity with the
free-wheeling introduction to
universal constructions in algebra and topology given, in textbook style,
by Bergman \cite{Ber}, in particular Section~3.17.


\section{The  Bohr compactification of a structure}\label{natBohr}

The Bohr compactification has a honourable place in the theory of topological groups, and  has
 important connections
with harmonic analysis and
 almost periodic functions.
 For background on the construction in this context,
and on its applications,  see for example \cite{Ho64,DFH10}.
The ideas were extended
to certain other classes of algebras  with compatible topology; see for example \cite{Ho64}, and the wide-ranging survey by Hart and Kunen \cite{HK99}.
We warn once again, however, that the term
`structure'  is used in a narrower sense in \cite{HK99} than in the present paper.
   In the former the setting is provided by a first-order language $\mathcal{L}$ with
operation symbols and equality but
without other relation  symbols; the authors suggest that the
theory would be `a little messier' if $\mathcal{L}$  were to include predicates
(see  \cite[2.1 and 2.3.13]{HK99}).
This is in sharp contrast to our treatment.
We work in a context that encompasses algebraic structures, purely relational structures, and hybrid structures within a common framework.
We will consider the Bohr compactification of these more general structures and connect it, where possible, to the natural extension.
Hart and Kunen make no a priori assumption that the classes of structures
with which they deal are varieties or prevarieties, though this is the case
with their most significant examples.

Within the theory of compactifications of topological algebras, or of more
general types of topological structures,
an important special case arises when one restricts to  the situation in which
the objects being compactified carry the discrete topology, or equivalently
no topology.   This is the case on which
we shall exclusively focus.
In most of our examples, the Bohr compactifications will be zero-dimensional.

We recall that a topological space~$X$ is said to be \emph{zero-dimensional} if it has
a basis of clopen sets.
This is a convenient point at which to draw attention to the alternative formulations
of the concept of zero-dimensionality in the context of compact
spaces. A compact Hausdorff space $X$ is zero-dimensional if and only if it is
a \emph{Boolean space}
in the sense that the clopen sets separate the points (that is, if it is \emph{totally
disconnected}).
For brevity we shall usually adopt  the term Boolean space subsequently.

Our task in this section is to set up the definition of the Bohr compactification,
in either variant, in the context of structures.  First we need to specify precisely
what we mean by a (topological) structure.

\begin{definition}\label{def:structure}
A \emph{structure}
$\A = \langle A;G^\A, H^\A,R^\A\rangle$ is a set $A$ equipped with a set $G^\A$ of finitary total operations, a set $H^\A$ of finitary partial operations and a set $R^\A$ of finitary relations. If $H^\A$ is empty we refer to $\A$ as a \emph{total structure},
if both  $G^\A$ and $H^\A$ are empty we refer to $\A$ as
a
\emph{purely relational structure} and if both $H^\A$ and $R^\A$ are empty we refer to $\A$ as an \emph{algebra}. A  \emph{structure with topology}
$\A = \langle A;G^\A, H^\A,R^\A,\Tp^\A\rangle$ is simply a structure equipped
with a topology~$\Tp^\A$, 
and $\A^\flat \coloneqq  \langle A;G^\A, H^\A,R^\A\rangle$ will denote its underlying structure.
We say that the topology is \emph{compatible} with the 
underlying structure 
if the relations in $R^\A$ and the domains of the partial operations in $H^\A$ are topologically closed and the operations in $G^\A$ and the partial operations in $H^\A$ are continuous; when this holds we refer to $\A = \langle A;G^\A, H^\A,R^\A,\Tp^\A\rangle$ as a \emph{topological structure}
(of signature $(G, H, R)$).
Given a 
structure $\M$ with a compatible topology~$\Tp$,
we denote by $\M_\T$ the topological structure obtained by endowing $\M$ with the 
topology~$\Tp$.
We shall sometimes use a superscript
$\Tp$ rather than a subscript to avoid bracketing;  
 for example, we write $\M_1^\Tp$ rather than $(\M_1)_\Tp$.
\end{definition} 

Our principal concern will be with total structures, but we do not disallow partial operations until this is necessary.
A class of structures will always be converted into a category by adding all
homomorphisms as morphisms of the category, and similarly, for a class of structures with topology, the morphisms of the corresponding category will be the continuous homomorphisms.

\begin{definition}\label{def:Bohr} 
Assume that~$\CA$ is a class of structures.
Then we may consider
both  the category $\CA^\mathrm{ct}$ of compact
Hausdorff topological structures having $\CA$-reducts and  its full subcategory $\CA^\mathrm{Bt}$ consisting of those compact topological structures whose topology is zero-dimensional.
The \emph{Bohr compactification} of~$\A$, denoted $b(\A)$, is required to be a member of $\CA^{\mathrm{ct}}$ into which $\A$ embeds as a structure, via an embedding we denote by $\iota_\A$, with the property that the closed substructure of $b(\A)$ generated by $\iota_\A(A)$ is $b(\A)$ itself.  If the signature of the structures includes no partial operations, so $H = \varnothing$, then we simply require $\iota_\A(A)$ to be topologically dense in $b(\A)$.
The compact topological structure $b(\A)$ is required to
satisfy, and is uniquely determined, up to
a $\CA^{\mathrm{ct}}$-isomorphism,
by  the following universal mapping property:
\begin{quote}
given any compact Hausdorff  structure $\B\in \CA^{\mathrm{ct}}$ and any $\CA$-morphism $g \colon \A \to
\B^\flat$, there exists a unique $\CA^{\mathrm{ct}}$-morphism $h\colon b(\A) \to \B$
such that $h \circ \iota_\A = g$.
\end{quote}
Replacing `Hausdorff' by `zero-dimensional'
and $b(\A) $ by $b_0(\A)$ throughout, so   working within the realm of
Boolean-topological structures, we obtain the \emph{zero-dimensional Bohr
compactification} $b_0(\A)$ of $\A$.
(Henceforth all compact topological spaces will be assumed to be Hausdorff.)
\end{definition}

It is a very simple exercise to check
that the (zero-dimensional) Bohr compactification is uniquely determined
and that the specification in terms of a universal mapping property agrees with that given in \cite{HK99}. If $\CA$ is closed under forming substructures, then, in the universal mapping property, the $\CA^{\mathrm{ct}}$-morphism $h$  is uniquely determined by $g$, for all $\CA$-morphisms~$g$, if and only if the closed substructure of $b(\A)$ generated by $\iota_\A(A)$ is $b(\A)$ itself---the argument is completely standard, using only the universal mapping property and the fact that an equaliser of two continuous homomorphisms forms a closed substructure.

Thus for $\A \in \CA$ both $b(\A)$ and $b_0(\A)$ are indeed defined
and are characterised by their respective universal mapping properties.
If it happens that $b(\A)$ is in fact a Boolean-topological structure, then $b(\A) = b_0(\A)$, since $b(\A) $ satisfies the universal property  characterising $b_0(\A)$. Of course, the universal mapping properties defining $b(\A)$ and $b_0(\A)$ say precisely that $b\colon \CA \to \CA^\mathrm{ct}$ and $b_0\colon \CA \to \CA^\mathrm{Bt}$ are reflections, that is, they are left adjoint functors to the natural forgetful functors, provided $b(\A)$ and $b_0(\A)$ exist, for all $\A\in \CA$.

\section{The natural extension of a structure}\label{natext}

In this section we introduce, in the context of 
CT-prevarieties of structures, the natural
extension which plays a central and unifying role in this paper.
The theory we shall present principally concerns
categories of the two forms:
\[
\CA = \ISP(\CM)  \quad \text{and} \quad \CA_\Tp = \IScPnp(\CM_\Tp).
 \]
Here and below $\CM$ is a set of structures of common
signature; $\CM$ is not required to be finite  (but  in the examples 
we shall give $\CM$ will contain only a single structure).
The \emph{prevariety} $\CA \coloneqq \ISP(\CM)$ generated by $\CM$ is the class of isomorphic copies of non-empty substructures of products
of structures in~$\CM$, where products are structured coordinatewise.  Extending the usage in \cite[Section~2]{nisp},
if all of the structures in $\CM$ are finite
we shall refer to the class $\CA$ as an \emph{internally residually finite prevariety
{\upshape(}of structures{\upshape)}}  or 
IRF-\emph{prevariety} for short.
We shall assume that each $\M$ in $\CM$ has a fixed associated compact topology $\Tp$ that is compatible with $\M$ and we denote the corresponding topological structure by $\M_\Tp$.
Let $\CM_\Tp \coloneqq \{\, \M_\Tp \mid \M \in \CM\,\}$; then  the \defn{topological prevariety} $\CA_\Tp \coloneqq  \IScPnp (\CM_\Tp)$
generated by $\CM_\Tp$ consists of isomorphic copies of non-empty topologically closed substructures of products of members of~$\CM_\Tp$.

\begin{remark}\label{rem:zero-one} 
If we prefer we may replace the class operator $\mathsf P$ by $\mathsf P^+$, thereby excluding the empty indexed product. Similarly, we can allow the possibility of including the empty structures in both $\CA$ and $\CA_\Tp$ by replacing $\mathsf S$ and $\mathsf S_{\mathrm{c}}$ with the operators $\mathsf S^0$ and $\mathsf S_{\mathrm{c}}^0$ that include empty substructures (when the signature does not include nullary operations). We have chosen one of the four
possibilities as our primary setting, but will have need of several of the others along the way.
  All of the theory presented below carries over to the other three with trivial changes. To avoid a proliferation of names, we shall refer to each of $\IScPnp(\CM_\Tp)$, $\IScP(\CM_\Tp)$, $\ISOcPnp(\CM_\Tp)$ and $\ISOcP(\CM_\Tp)$ as the topological prevariety generated by
$\CM_\Tp$ as it will always be clear from the context which is intended.
\end{remark}

In the 
case of an IRF-prevariety $\ISP(\CM)$, each $\M \in \CM$ is finite and hence the topology $\Tp$ associated with $\M$ is discrete and is the unique topology  making
$\M _\Tp$ compact Hausdorff (in fact zero-dimensional). 
The objects in $\CA_\Tp = \IScPnp(\CM_\Tp)$ are 
then Boolean-topological structures; hence $\CA_\Tp \subseteq \CA^{\mathrm{Bt}}$.

Now let $\CA$ be a CT-prevariety of structures, so $\CA = \ISP(\CM)$ for a set $\CM$ of structures in $\CA$ each having an associated compatible compact topology. Let $\CA_{\Tp} \coloneqq \IScPnp(\CM_{\Tp})$ be the associated topological prevariety. We are  ready to extend to 
CT-prevarieties of
structures the concept of a \emph{natural extension} which was introduced for
IRF-prevarieties of
 algebras in \cite[Section 3]{nisp}. Let  $\A \in \CA$ and define
\[
X_\A \coloneqq  \dotbigcup\{\, \CA(\A, \M) \mid \M\in \CM\,\}.
\]
Further, let $\Y_x \coloneqq  \M_{\Tp}$, for each $\M \in \CM$ and $x\in \CA(\A, \M)$,  that is, $\Y_x$ is the codomain $\M$ of the map $x$ with the discrete topology $\Tp$ added.
The homomorphism
\[
\esubA \colon  \A \to \prod\big\{\, \Y_x \mid x\in X_\A\,\big\}
\]
given by evaluation, $\esubA(a)(x) \coloneqq  x(a)$,
for all $a \in A$ and $x\in X_\A$, is an embedding of structures
since $\A\in \ISP(\CM)$.  We also observe
that   $\prod\big\{\, \Y_x \mid x\in X_\A\,\big\} \in \CA_{\Tp}$.

\begin{definition}\label{def:natext}
Let $\CA = \ISP(\CM)$ be a CT-prevariety of structures generated by a set $\CM$ of structures each with a fixed associated compact topology and let $\A \in \CA$. Then the topologically closed substructure generated by $\esub\A(\A)$ in $\prod\{\, \Y_x\mid x\in X_\A\,\}$ is said to be the \defn{natural extension} $n_{\CA}(\A)$ of $\A$ in $\CA_\Tp$ \textup(relative to~$\CM_\Tp$\textup).
\end{definition}

We notice that, in the case of total structures, the natural extension $n_{\CA}(\A)$  coincides with
the topological closure of
$\esub\A(\A)$ in
$\prod\{\, \Y_x\mid x\in X_\A\,\}$.
For a CT-prevariety of structures, $\CA = \ISP(\CM)$,
we have constructed a map $\A \mapsto n_{\CA}(\A)$ from $\CA$ into
$\CA_{\Tp}$. Though this seems to depend upon the choice of the generating   set~$\CM$ of structures for the prevariety~$\CA$, we shall show later that
$n_{\CA}(\A)$ is independent of the choice of the generating set~$\CM$ of the prevariety $\CA$ in the case of an IRF-prevariety. More  generally, the natural  extension on a CT-prevariety $\CA= \ISP(\CM)$ is  independent of the chosen generating set $\CM_\Tp$ of the associated topological prevariety
(Corollary~\ref{nat-indep}).

The map $n_{\CA}$ is defined on morphisms just  as in \cite{nisp}.
Let $u \colon  \A \to \B$ be a morphism with $\A,\B \in \CA$. For $y \in X_\B$ we have  $y \circ u \in X_\A$, and
 for each $y \in X_\B$ we have the map
\[
u_y \colon  \prod\{\, \Y_x \mid x\in
X_\A\,\} \to \Y_y
\]
defined by $u_y(f) \coloneqq  f(y \circ u)$. Further, $\Y_y =
\Y_{y \circ u}$ and $u_y$ (as the
projection at $y \circ u$) is continuous. The map
\[
\widehat u \colon  \prod\{\, \Y_x \mid x\in X_\A\,\} \to
\prod\{\, \Y_y \mid y\in X_\B\,\}
\]
is then defined as the natural product map,
that is,
 \[
(\widehat u (f))(y) \coloneqq  u_y(f)= f(y \circ u), \quad
\text{for } f \in  \prod\{\, \Y_x \mid x\in X_\A\,\} \text{ and }
y \in X_\B.
\]
As each $u_y$ is continuous, $\widehat u$ is continuous  too. The following properties of $\widehat u$ are similar to those presented in \cite[Lemma 3.1]{nisp}. While the proof of the first one is analogous to the proof in the case of algebras, the second one requires slightly more careful definition chasing.

\begin{lemma}\label{lem:1}\
Let $u \colon  \A \to \B$ be a morphism with $\A,\B \in \CA$.
\begin{newlist}
 \item[\upshape(i)]
 $\widehat u\circ \esubA = \esub \B\circ u$, and consequently,
$\widehat u(\esubA (A)) \subseteq \esub \B(B)$.

 \item[\upshape(ii)]
$\widehat u(n_\CA(\A)) \subseteq n_\CA(\B)$.
\end{newlist}
\end{lemma}

\begin{proof}To prove (i) we proceed as follows.  Let $a \in A$.
Then, for all $y \in X_\B$, 
\begin{multline*}
(\widehat u \circ\esubA) (a)(y) =
\widehat u(\esubA (a))(y)
 \coloneqq  u_y(\esubA(a))
 = \esubA(a)(y \circ u)\\
 = y(u(a)) = \esub\B(u(a))(y) = (\esub\B\circ u)(a)(y).
\end{multline*}
Hence $\widehat u\circ \esubA = \esub \B\circ u$, and it follows at once that $\widehat u(\esubA (A)) \subseteq \esub \B(B)$.

For (ii), we first note that, by (i), 
\[
\esubA(A) \subseteq \widehat u^{-1}(\widehat u(\esubA (A)))  \subseteq
\widehat u^{-1}(\esub\B (B))  \subseteq \widehat u^{-1}(n_\CA(\B)).
\]
Since $\widehat u^{-1}(n_\CA(\B))$ is a closed substructure of $\prod\{\, \Y_x \mid x\in X_\A\,\}$, it follows that $n_\CA(\A)\subseteq \widehat u^{-1}(n_\CA(\B))$, and thus $\widehat u(n_\CA(\A)) \subseteq \widehat u(\widehat u^{-1}(n_\CA(\B)))\subseteq n_\CA(\B)$.
\end{proof}

For structures  $\A,\B \in \CA$ and a
morphism $u \colon  \A \to \B$, we define a continuous
morphism
$n_\CA(u) \colon  n_\CA(\A) \to n_\CA(\B)$
by
$n_\CA(u) \coloneqq  \widehat u\rest {n_\CA(\A)}$, and
the first part of the following proposition follows by
a routine calculation. The second
is a consequence of Lemma~\ref{lem:1} where
 ${}^\flat \colon  \CA_{\Tp} \to \CA$ denotes the natural forgetful functor.

\begin{proposition}\label{prop:functor}

\begin{newlist}

\item[\upshape (i)] $n_\CA \colon \CA \to \CA_\Tp$ is a well-defined functor.

\item[\upshape (ii)] $e \colon \id_\CA \to n_\CA^\flat$ is a natural transformation, where $n_\CA^\flat \coloneqq  (n_\CA)^\flat \colon \CA \to \CA$.

\end{newlist}

\end{proposition}

Lemma~\ref{lem:alt} below presents an alternative view of
the natural extension of a structure in the CT-prevariety $\CA = \ISP(\CM)$.
We shall need this result
shortly in order
 to prove that the natural extension functor is a reflection.  The lemma  extends to the setting of
CT-prevarieties of
structures an analogous result for  
IRF-prevarieties of
algebras given in \cite{nisp}.

We adopt the same notation as above.
The product $\prod\big\{\, \Y_x \bigm| x\in X_\A\,\big\}$,  the codomain of the map $\esubA$, may be viewed as an iterated product
\[
\prod  \big\{ \, \prod \{\, Y_x \bigm| x\in \CA(\A, \M)\,\big\}
\bigm| \M \in \CM\, \big\}.
\]
We then write
$\esubA (a) (\M)(x) = x(a)$,  for any fixed $a \in A$ and for $\M \in \CM$
and $x \in \CA (\A, \M)$,
and refer to  each $\esubA(a) $ as a \defn{multisorted evaluation map}.
 We have
\[
\esubA \colon  \A \to \prod \big\{ \, \M_\Tp^{\CA(\A,\M)} \bigm| \M \in
\CM \, \big\}.
\]
The set $\CA(\A,\M)$ can be regarded as a closed subspace
of the topological product $\M_\Tp^A$, in which case we
denote it by  $\CA(\A,\M)_\Tp$.
(Notice we are not claiming that
$\CA(\A,\M)_\Tp \in \CA_\Tp$; in general it is not a
substructure
of  $\M_\Tp^A$.) It now
makes sense to consider the set
$\mathrm{C}(\CA(\A,\M)_\Tp,\M_\Tp)$ of continuous maps  from
$\CA(\A,\M)_\Tp$  into $\M_\Tp$. As the map
\[
\esub\A(a)(\M) \colon  \CA(\A,\M)_\Tp  \to \M_\Tp
\]
 is continuous, for all $\M \in
\CM$,
 we can restrict the codomain of $\esub \A $ and
write
\[
\esub\A \colon  \A \to \prod \big\{\, \mathrm{C}
(\CA(\A,\M)_\Tp,\M_\Tp) \bigm|
\M \in \CM \, \big\}.
\]

\begin{lemma}\label{lem:alt}  Let $\CA=\ISP(\CM)$ be a {\upshape CT}-prevariety
of structures and let $\A \in \CA$.
Then the  natural extension $n_{\CA}(\A)$ is the closed substructure generated by $\esub\A(\A)$ within the product $\prod \big\{\, \mathrm{C}(\CA(\A,\M)_\Tp,\M_\Tp) \bigm| \M \in \CM \big\}$.
\end{lemma}

We can provide a quite explicit, if unwieldy, description of the elements of
the natural extension in the context of an IRF-prevariety of structures. This 
generalises the description given by \cite[Theorem~4.1]{nisp} and is proved in the same way.  We present this in the single-sorted case
(so that $|\CM|=1$) since this covers our future needs in this paper and
simplifies the statement; 
a multi-sorted version could be obtained, as in \cite{nisp}.

\begin{proposition} \label{multi-pres:noduality}
Let $\M = \langle M; G, R\rangle $ be a finite total structure, let $\CA \coloneqq  \ISP{(\M)}$, let $\A$ belong to~$\CA$, and let $b \colon \CA(\A, \M) \to M$ be a map. Then the following conditions  are equivalent:
\begin{newlist}
\item[\upshape(i)] $b$ belongs to $n_\CA(\A)$, that is,
$b$ belongs to the topological closure of $\esub\A(\A)$ in
$\M_\Tp^{\CA(\A,\M)}$;

\item[\upshape(ii)] $b$ is locally an evaluation,
that is, for every finite subset $Y$ of $\CA(\A, \M)$,
there
exists $a\in A$ such that $b(y) = y(a)$, for all $y\in Y$;

\item[\upshape(iii)] $b$ preserves every finitary
relation on $M$ that forms a substructure of the appropriate power of~$\M$;

\item[\upshape(iv)] $b$ preserves every finitary
relation on
$M$
of the form
\[
r_F \coloneqq  \{\, (x_1(a), \dots, x_n(a)) \mid a \in A\,\},
\]
where $F = \{x_1, \dots, x_n\}$ is a finite subset of $\CA(\A, \M) $.
\end{newlist}
\end{proposition}

  We now revert to the assumption that  $\CA=\ISP(\CM)$ is a CT-prevariety of structures.
With the construction of the natural extension in place we  are ready to move on to
establish its key properties.

We wish to prove that the natural extension functor is a reflection. To
do this we exploit the alternative  description of  the natural extension
given in Lemma~\ref{lem:alt}.
This theorem, proved for the algebra case in \cite[Proposition~3.4]{nisp}, was not exploited in that paper.  Here, extended to structures and slightly rephrased, it will play an important role (see Section~\ref{sec:BviaN}).
We note that the statement in Theorem~\ref{lem:4}(i)  is slightly stronger than
asking for $\B^\flat$ to be a
retract of $n_\CA(\B^\flat)^\flat$ in $\CA$.

\begin{theorem}\label{lem:4}  Let $\CA$ be a {\upshape CT}-prevariety of
structures.
\begin{newlist}
\item[\upshape (i)]
For each $\B\in \CA_\Tp$
there exists
 a continuous homomorphism $\gamma
\colon n_\CA(\B^\flat) \to \B$ with $\gamma \circ \esub {\B^\flat} = \id_B$.
\item[\upshape (ii)]
The natural extension functor $n_\CA \colon  \CA \to \CA_\Tp$ is a
reflection of $\CA$ into the {\upshape(}non-full{\upshape)}
subcategory~$\CA_\Tp$.
Specifically,
for each  $\A\in \CA$,
each  $\B\in \CA_\Tp$ and every  homomorphism $g \colon  \A\to\B^\flat$,
there exists a unique continuous homomorphism $h \colon  n_\CA(\A) \to \B$ with
$h\circ \esubA  = g$.
\end{newlist}
\end{theorem}

\begin{proof}  Consider (i).
Let $\B\in \CA_\Tp$ and consider the natural map
\[
c \colon  \B \to \prod \big\{ \, \M_\Tp^{\CA_\Tp(\B,\M_\Tp)} \bigm| \M \in
\CM \, \big\}  \quad\text{given by}\quad c(b)(\M)(x) \coloneqq  x(b),
\]
for all $b \in B$ and $x\in \CA_\Tp(\B,\M_\Tp)$. Since $\B\in
\CA_\Tp$, the map $c$ is a continuous embedding. Let
\[
\pi \colon  \prod \big\{ \, \M_\Tp^{\CA(\B^\flat,\M)} \bigm| \M \in \CM \,
\big\} \to\prod \big\{ \, \M_\Tp^{\CA_\Tp(\B,\M_\Tp)} \bigm| \M \in
\CM \, \big\}
\]
be the obvious projection.
Clearly, $\pi\circ \esub {\B^\flat} = c$ and $\pi$ maps $\esub{\B^\flat} (B)$
bijectively to~$c(B)$. Since $\esub {\B^\flat}(B) \subseteq \pi^{-1}(\pi(\esub {\B^\flat}(B))) = \pi^{-1}(c(B))$, and since $\pi^{-1}(c(\B))$ is a closed substructure of $\prod \big\{ \, \M_\Tp^{\CA(\B^\flat,\M)} \bigm| \M \in \CM \,\big\}$, we have $n_\CA({\B^\flat})\subseteq \pi^{-1}(c(\B))$, whence $\pi(n_\CA({\B^\flat}))\subseteq c(\B)$.

Hence we can restrict both the domain and the codomain of $\pi$ and define
\[
\rho \coloneqq  \pi\rest{n_\CA({\B^\flat})}\colon  n_\CA({\B^\flat}) \to c(\B).
\]
 Finally, define $\gamma \coloneqq  c^{-1} \circ \rho$. Then we have
\[  
\gamma \circ \esub{\B^\flat} = c^{-1} \circ \rho \circ \esub{\B^\flat} =
c^{-1} \circ c = \id_B,
\]
completing the proof of (i).

Now consider (ii).   We first prove the uniqueness of the continuous homomorphism~$h$. Assume that  continuous homomorphisms $h, h' \colon  n_\CA(\A) \to \B$ satisfy $h\circ \esubA = h'\circ \esubA
= g$. Then the equaliser $\Y \coloneqq \eq(h, h')$ is a closed substructure of $\prod\big\{\, \Y_x \bigm| x\in X_\A\,\big\}$ containing $e_\A(A)$ and hence $\Y = n_\CA(\A)$; it follows at once that $h = h'$.

To prove the existence assertion, we apply (i) to find
$\gamma
\colon n_\CA(\B^\flat) \to \B$ with  $\gamma \circ \esub {\B^\flat} = \id_B$.
Note that $n_{\CA}(g) \colon n_{\CA}(\A) \to n_{\CA}(\B^\flat)$ is a continuous homomorphism with  $n_{\CA}(g) \circ \esubA = \esub {\B^\flat}\circ g$.
Now take $h = \gamma \circ n_{\CA}(g)$.
\end{proof}

It is very easy to check that the definition of $n_\CA(\A)$ requires only that
$\A$ be a structure of the appropriate
signature.
Proposition~\ref{prop:functor} and Theorem~\ref{lem:4} then show that $n_\CA$ provides a reflection functor from the category of all structures of the appropriate type into $\CA_\Tp$.

 The fact that  Theorem~\ref{lem:4}
 supplies
a reflection has the following important corollary.

\begin{corollary}\label{cor:leftadjoint}
For each {\upshape CT}-prevariety of structures $\CA$, the
functor $n_{\CA} \colon
\CA \to \CA_\Tp$ is left adjoint to the functor $^\flat \colon \CA_\Tp \to \CA$ forgetting the topology.
\end{corollary}

Since the left adjoint to the forgetful functor is unique and depends only upon $\CA$ and the subcategory $\CA_\Tp$, we obtain the following important consequence for the natural-extension perspective we adopt in the remainder of the paper.

\begin{corollary}\     \label{nat-indep}
\begin{newlist}
 \item[\upshape(i)] 
Let $\CM$ and $\CM'$ be sets consisting of structures each of which has a fixed associated compact topology.
Define $\CA = \ISP(\CM)$ and assume that $\IScPnp(\CM_\Tp) = \IScPnp(\CM_\Tp')$.
Then $\CA = \ISP(\CM')$ and, for all $\A\in \CA$, the natural extensions
 of $\A$ relative to $\CM$ and relative to $\CM'$ agree.

 \item[\upshape(ii)] 
Let $\CA$ be an {\upshape IRF}-prevariety of structures.
Then, for each $\A\in \CA$, the natural extension $n_\CA(\A)$ of $\A$ is independent of the set $\CM$ of finite structures chosen to generate~$\CA$.
\end{newlist} 
\end{corollary}

\begin{proof}
By Corollary~\ref{cor:leftadjoint}, we only need 
to
see that our assumptions guarantee that $\ISP(\CM) = \ISP(\CM')$, so that the categories are not changed when we pass from $\CM$ to $\CM'$.
We have $\IScPnp(\CM_\Tp) = \IScPnp(\CM_\Tp')$ by assumption. 
It follows that $\CM_\Tp \subseteq \IScPnp(\CM_\Tp')$, whence 
$\CM \subseteq \ISP(\CM')$ and so $\ISP(\CM) \subseteq \ISP(\CM')$. By symmetry we have the reverse inclusion and so $\ISP(\CM) = \ISP(\CM')$. This proves (i).

To prove (ii), it suffices to show that, if $\CM$ and $\CM'$ consist of finite structures, then $\ISP(\CM) = \ISP(\CM')$ implies that $\IScPnp(\CM_\Tp) = \IScPnp(\CM_\Tp')$. Assume that $\ISP(\CM) = \ISP(\CM')$. Since the topologies involved are discrete, it follows easily from this that $\CM_\Tp \subseteq \IScPnp(\CM_\Tp')$ and $\CM_\Tp' \subseteq \IScPnp(\CM_\Tp)$, whence $\IScPnp(\CM_\Tp) = \IScPnp(\CM_\Tp')$.
\end{proof}


\section{The Bohr compactification versus the natural extension:\\ the role of standardness}\label{sec:BviaN}

Our goal in this section is to compare Bohr compactifications to the natural extension in situations where the latter is defined.

Consider until further notice the situation in which we have 
a CT-prevariety $\CA = \ISP(\CM)$ of structures and its associated topological prevariety $\CA_\Tp= \IScPnp(\CM_\Tp)$. Note that we always have $\CA_\Tp\subseteq \CA^{\mathrm{ct}}$, and if the topologies on the members of $\CM$ are Boolean, in particular if $\CA$ is an IRF-prevariety, then we have $\CA_\Tp\subseteq \CA^{\mathrm{Bt}}$. 
Observe that it is the category $\CA_\Tp$ that appears in Theorem~\ref{lem:4}, rather than either of the potentially larger
categories $\CA^{\mathrm{Bt}}$ and $\CA^{\mathrm{ct}}$.

The Bohr compactification (in both zero-dimensional and compact Hausdorff
versions) and the natural extension of a structure are characterised by universal mapping properties;
Definition~\ref{def:Bohr} and Theorem~\ref{lem:4}.
Thus we have reflection functors into three possibly different categories
(see Figure~\ref{fig:new}):
\begin{newitemize}
\item   
the natural extension functor, providing a reflection into $\CA_\Tp$;
\item 
 the zero-dimensional Bohr compactification functor $b_0$, giving a reflection into the category
$\CA^{\mathrm{Bt}}$;
\item  
the Bohr compactification functor $b$, giving a reflection into the category $\CA^{\mathrm{ct}}$.
\end{newitemize}
In each case the functor is uniquely determined by its characteristic universal mapping property.
We recall from Section~\ref{sec:intro} that 
\emph{strong coincidence} occurs when two of the functors $n_{\CA}$, $b_0$ and~$b$ on~$\CA$ coincide because their codomains are the same, and that
\emph{weak coincidence} arises when the codomain
categories are different but the image of the functor into the larger 
of the categories lies in the smaller one.
The following proposition is immediate.  

\begin{proposition} \label{prop:coincide}
Let $\CA=\ISP(\CM)$
be a {\upshape CT}-prevariety of structures
and define $\CA_\Tp\coloneqq \IScPnp(\CM_\Tp)$.
Then the following statements hold.
\begin{newlist}
\item[{\upshape (i)}]  If $\CA_\Tp = \CA^{\mathrm{ct}}$, then
$b(\A) = n_{\CA}(\A)$ for each $\A \in \CA$.
\item[{\upshape (ii)}] If 
$\CA_\Tp = \CA^{\mathrm{Bt}}$, then
$b_0(\A) = n_{\CA}(\A)$ for each $\A \in \CA$.

\item[{\upshape (iii)}] Let $\A \in\CA$.  Then
\begin{newlist}
\item[{\upshape (a)}]
 $b_0(\A) = n_{\CA}(\A)$ if and only if $b_0(\A) \in \CA_\Tp$;
\item[{\upshape (b)}]   $b(\A) = b_0(\A) $ if and only if
$b(\A)
\in \CA^{\mathrm{Bt}}$.
\end{newlist}
\end{newlist}
\end{proposition}

Suppose that $\CA$ is such that we have an explicit description
of each $n_{\CA}(\A)$.  
Then strong or weak coincidence of $b_0$  with $n_{\CA}$
allows us to describe ~$b_0$, and likewise with~$b$ in place of~$b_0$.
 To exploit the above observations we need to know more  about~$\CA_\Tp$.
We are fortunate that a wealth of information is already available, or is easy to obtain, in the special case that most interests us: that in 
which $\CM$ contains a single structure~$\M$.  

In the case that $\M$ is finite, the assumption $\CA_\Tp = \CA^{\mathrm{Bt}}$ in Proposition~\ref{prop:coincide}(ii) is exactly the condition that the topological prevariety $\CA_\Tp$ is \emph{standard},
 in the sense that
$\CA_\Tp $ consists precisely of the structures which are Boolean-topological models of the quasi-equations defining $\CA$.
We then have the following theorem concerning  strong coincidence of 
$b_0$ and $n_{\CA}$.

\begin{theorem} \label{thm:NatIsBohr}
Let $\M$ be a finite structure, define $\CA\coloneqq  \ISP(\M)$
and assume that the associated
topological prevariety $\CA_\Tp\coloneqq  \IScPnp(\M_\Tp)$
is standard. Then, for every $\A \in \CA$,  the
zero-dimensional Bohr compactification $b_0(\A)$ of the structure $\A$ coincides with its natural extension $n_{\CA}(\A)$.
\end{theorem}

Our linkage of standardness to the coincidence of structures characterised by universal mapping properties is new. However the notion of
standardness has received a lot of attention in its own right, 
 principally in the case that
$\M$ is an algebra, but to a limited extent when~$\M$ is a structure
(we consider the latter case later).
 The systematic study of standardness
of a topological prevariety
$\ISOcP( \M_\Tp)$ was initiated
in \cite{CDHPT,CDFJ}.
In these papers $\M$ is taken to be a finite structure (not necessarily  an algebra and not necessarily total).   
While the theory of standardness was developed for topological prevarieties of the form $\ISOcP( \M_\Tp)$, the results apply with almost no change to all four settings described in Remark~\ref{rem:zero-one}; in particular they apply to the class $\IScPnp(\M_\Tp)$ of interest here.

There are interesting and substantial results available  `off-the-peg'
when $\M$ 
is a finite algebra.  Assume this,   and assume moreover that
$\ISP(\M) = \HSP(\M)$ so that $\ISP(\M)$ is a variety. 
The principal general
result of~\cite{CDFJ},
the FDSC-HSP Theorem, reveals that a rather natural algebraic condition ensures standardness of $\IScPnp(\M_\Tp)$.
This property---having finitely determined syntactic congruences---holds in particular if $\HSP(\M)$ has the more familiar property
of having  definable principal congruences (see \cite[Section~2]{CDFJ} for the definitions and discussion, and~\cite[Theorem~2.13]{CDJP} for an extension of the FDSC-HSP Theorem to total structures). In some cases the FDSC condition will hold for an entire variety of algebras, and hence for its finitely generated subvarieties; in others, in particular lattices,
restriction to finite generation is critical if FDSC is to hold. 
The FDSC-HSP Theorem implies that the topological prevariety $\IScPnp(\M_\Tp)$
is standard in each of the following cases:  $\M$ is a finite Boolean algebra, distributive lattice or implication algebra, or $\M$ is a finite group, semigroup, ring, lattice, Ockham algebra, or unary algebra such that $\HSP(\M) = \ISP(\M)$.
This catalogue  
 is not exhaustive. 
For additional  examples, and verifications of the claims above, 
see \cite[Section~6]{CDFJ} 
 and also~\cite{DT}.
We should however warn that standardness is a subtle property in general,
and can fail:
there exist finite algebras $\M$ for which $\IScPnp(\M_\Tp)$ is not standard.
An example is given in \cite[Example~4.3]{CDJP} in which $\M$ is a 10-element modular lattice. Further insight into when and why standardness occurs is provided by \cite{CDJP,CDMM} and, for structures, \cite[Section~3]{D06}.
On the positive side,  then, Theorem~\ref{thm:NatIsBohr} is rather widely applicable.  Moreover,  in many cases when it is, we shall confirm below that
the natural extension has an appealingly simple description, so that the zero-dimensional  Bohr compactification does too.

We draw attention to a well-known instance of non-standardness in the context of structures.

\begin{example}\label{ex:Stralka} 
Consider the category of ordered sets, $\CP = \ISOP(\twoT)$,
and the associated topological prevariety
$\CP_\Tp= \ISOcPnp(\twoT_\Tp)$ of Priestley spaces,
where ${\twoT = \langle \{0,1 \}; \leqslant\rangle}$ is the two-element chain.
Stralka \cite{Str80} exhibited two examples of Boolean spaces with a closed order relation that fail to be Priestley spaces, whence $\CP_\Tp$
is non-standard.
(For further analysis of this phenomenon, see \cite{BMM02}.)

We shall show in
Proposition~\ref{NachisNat} that $b_0$ and $\n_{\CP}$ do coincide 
(in fact  $b$ coincides with $n_{\CP}$ too).   
Here we have 
an instance of coincidence occurring in the weak sense but not in the  strong sense.
We deduce that
none of the  examples witnessing non-standardness of $\CP_\Tp$ belongs to
the image of the class~$\CP$ under~$b_0$ (or equivalently~$b$).
\end{example}

It is not always the case that the Bohr compactification and its
zero-dimensional variant coincide in the weak sense.
 This was established  for meet semilattices
by Hart and Kunen
\cite[Section~3.4]{HK99},
in particular
\cite[Corollary 3.4.13]{HK99},
drawing
on pioneering work by
Lawson (see
\cite{HK99} and \cite[Chapter~VI]{comp2} for references).
We present a proof which takes full advantage of the theory of continuous lattices, as presented in a mature form in \cite{comp2}, as well as results we have to hand.
Hart and Kunen work with non-unital meet semilattices.
For convenience we work with the variety $\CS$ of
unital meet semilattices.
When, as in \cite{HK99}, the unit~$1$ is not included in the signature,
one may pass to semilattices which do have~$1$; see \cite[p.~452]{comp2}.

We shall draw on
the Fundamental Theorem on Compact Totally Disconnected Semilattices
\cite[Theorem~VI-3.13]{comp2}.
In outline this asserts that  the objects of $\CS^\text{Bt}$ are those compact
topological unital meet semilattices that have small semilattices, meaning
that there exists a neighbourhood basis of subsemilattices at each point.
Moreover there is an isomorphism of categories  between
$\CS^{\text{Bt}}$ and
the category  of algebraic lattices equipped with the Lawson topology and with
the Lawson-continuous maps preserving~$1$ as morphisms;
these morphisms can alternatively be described
as the maps which preserve arbitrary
meets and directed joins.

\begin{theorem}
\label{thm:badsemilat}
Let $\CS$ be the class of unital meet semilattices. 
Then there exists $\A$ in~$\CS$ such that $b(\A) \ne b_0(\A)$.
\end{theorem}
\begin{proof}
There exists a compact topological unital meet semilattice $\B$
which fails to have small semilattices. See \cite[VI-4.5]{comp2}
for the statement, and the definitions and
lemmas which precede it for the construction. Let $\A = \B^\flat$ so that
$\A \in \CS$.
Suppose
 for a contradiction that $b(\A) = b_0(\A)$.
Then $b(\A)$ is a compact zero-dimensional  
unital meet semilattice
and hence is an algebraic lattice.   Moreover, $b(\A) = b_0(\A)
 =n_{\CS}(\A)$, by the standardness
of~$\CS_\Tp$.

By Theorem~\ref{lem:4}(i) there exists  a continuous homomorphism $\gamma \colon  n_{\CS}(\B^\flat)\to
\B$ such that $\gamma \circ e_{\A} = \id_B$. 
 This implies that~$\B$
is the  image under a continuous homomorphism of a compact (totally disconnected)
topological unital meet semilattice.  Since the domain of this map has small
semilattices, the same is true of the image, by \cite[VI-3.5]{comp2}
(or by the  specialisation of this result to the totally disconnected case).  It follows by the
Fundamental
 Theorem \cite[Theorem VI-3.13]{comp2} that $\B \in \CS^{\mathrm{Bt}}$.
 This provides the required contradiction.
\end{proof}

We now turn to the 
case in which $\M$ is an infinite structure and $\M_\Tp$ is a compact topological structure.

Standardness has been studied almost exclusively in the case where $\CA$ is a universal Horn class generated by its finite members.
Nevertheless, by analogy,
in the case that $\M_\Tp$ is an infinite compact topological structure,
it is natural to
define $\CA\coloneqq \ISP(\M)$ and
say that the topological prevariety $\CA_\Tp \coloneqq  \IScPnp(\M_\Tp)$ is \emph{compact-standard} if $\CA_\Tp = \CA^{\mathrm{ct}}$ and is \emph{Boolean-standard} if $\CA_\Tp = \CA^{\mathrm{Bt}}$. Since $\CA_\Tp \subseteq \CA^{\mathrm{ct}}$ always holds, it follows that $\CA_\Tp$ is compact-standard precisely when every compact
topological structure whose non-topological reduct is in $\CA$ can be embedded as a topological structure into a power of $\M_\Tp$. 
For example, the (highly non-trivial) fact that every compact
topological abelian group embeds into a power of the circle group $\mathbb T_\Tp$ tells us that the topological prevariety $\IScPnp(\mathbb T_\Tp)$ is the class of compact topological abelian groups (see~\cite[C, p.~241]{Pont})
and so is compact-standard. 
Similarly, if $\Tp$ is a Boolean topology, then $\CA_\Tp$ is Boolean-standard precisely when every Boolean topological structure whose non-topological reduct is in $\CA$ can be embedded as a topological structure into a power of $\M_\Tp$. We now give an example of a Boolean-standard prevariety with infinite generator.

\begin{example}[Semilattices with automorphism]\label{jezek}
We consider
the variety of semilattices with automorphism, as introduced by Je\v{z}ek
\cite{jez}; see \cite[Section~8]{DJPT} for further details. We let $\mathbf P$ have universe $2^{\mathbb Z}$, the set of all functions from the integers into the set $2 \coloneqq  \{0,1\}$. The meet operation  is defined  pointwise, relative to the two-element semilattice $\langle 2;\land \rangle$.  We let  $s \colon
\mathbb Z \to \mathbb Z $ be the successor function given by $s(i) \coloneqq  i+1$, for all
$i \in \mathbb Z$ and let~$\boldsymbol 0$ be the function on $\mathbb Z$ with constant value~$0$. We add the shift operations~$f$ and~$f^{-1}$ to the signature,
where $f(a) = a \circ s$ and $f^{-1}(a) = a \circ s^{-1}$
for $a\in 2^{\mathbb Z}$.
Thus $\mathbf P  = \langle \{0,1\}^{\mathbb Z}; \land, f,f^{-1}, \boldsymbol 0\rangle$.
Define $\CA  \coloneqq \ISP(\mathbf P)$ and let $\mathbf P_\Tp$ denote~$\mathbf P$ equipped with the product topology obtained from the discrete topology on $\{0,1\}$. We claim that $\CA_\Tp \coloneqq  \IScPnp(\mathbf P_\Tp)$ is Boolean-standard.

Consider a Boolean-topological algebra $\A$ having algebraic reduct
in~$\CA$.  We must show that~$\A$ embeds into a power of~$\mathbf P_\Tp$ via a continuous $\CA$-homomorphism.  The algebra
$\A^{\!\top}$, which denotes~$\A$  with a top element $\top$ adjoined as a topologically  isolated point, has a semilattice reduct which is a unital
meet semilattice.  The Fundamental Theorem for Compact Totally Disconnected Semilattices \cite[Theorem VI-3.13]{comp2} tells us that $\A^{\!\top}$ is an algebraic lattice and that its topology is the Lawson topology.  Extend~$f$ to $\A^{\!\top}$ by letting $f(\top) = \top$.  Then~$f$, so extended, is 
a semilattice homomorphism of $\A^{\!\top}$ 
preserving~$\top$. We define a family of maps
from $\A^{\!\top}$ into~$\mathbf P$ indexed by the compact elements~$k$
(excluding~$\top$) of $\A^{\!\top}$ as follows:
\[
h_k (x)(i) = 1 \iff f^i(k) \leq x.
\]
Then, for each fixed~$j \in \mathbb Z$, the set $\{ \, x \mid  h_k(x)(j) =1 \,\}$ is equal to $ {\uparrow} f^j(k)$. Now note that because the extended maps $f$, $f^{-1}$ and their iterates are semilattice automorphisms, and hence order-isomorphisms, of $\A^{\!\top}$, the element $f^j(k)$ is compact whenever~$k$ is.  It follows from properties of the Lawson topology on an algebraic lattice that each set $\{ \, x \mid  h_k(x)(j) =1 \,\}$
is clopen in~$\A^{\!\top}$ (see for example \cite[Exercise III-1.4]{comp2}).
This proves that the inverse image under $h_k$ of each member of a clopen
subbasis in $\mathbf P$ is clopen in $\A^{\!\top}$.  Since $\top$ is an isolated point with $h_k^{-1}(\top) = \{ \top\}$, the restriction
$h_k{\restriction}_\A$ is a continuous map of $\A$ into~$\mathbf P$. As shown in \cite[Proposition~1.1]{jez}, each $h_k$ is an $\CA$-homomorphism.

To show that $\A$ embeds into a power of $\mathbf P_\Tp$, it suffices
to show that the maps $h_k$ separate the points of~$\A$. Take $a \nleqslant b$ in~$\A$.  Since $\A^{\!\top}$ is an algebraic lattice, there exists a compact element~$k\ne \top $ of~$\A^{\!\top}$ with $k \leq a$ and $k \nleqslant b$.   Then $h_k$ separates~$a$ and~$b$.

It follows that $\IScPnp(\mathbf P_\Tp)$ is Boolean-standard, as claimed.
\end{example}


We now present the infinite-generator version of Theorem~\ref{thm:NatIsBohr}.

\begin{theorem} \label{thm:NatIsBohrInf}
Let $\M$ be an infinite structure and let $\Tp$ be a compact
topology on $M$ that is compatible with $\M$.
Define $\CA=\ISP(\M)$ and  $\CA_\Tp \coloneqq  \IScPnp(\M_\Tp)$.
\begin{newlist}
\item[{\upshape (i)}]
If the topological prevariety
$\CA_\Tp$ is compact-standard, then, for every $\A \in \CA$,  the
Bohr compactification $b(\A)$ of the structure $\A$ coincides with its natural extension $n_{\CA}(\A)$.

\item[{\upshape (ii)}] If the topology $\Tp$ is Boolean, and the topological prevariety $\CA_\Tp$ is Boolean-standard, then, for every $\A \in \CA$,  the zero-dimensional Bohr compactification $b_0(\A)$ of the structure $\A$ coincides with its natural extension $n_{\CA}(\A)$.
\end{newlist}
\end{theorem}

This is an opportune point at which to make some background comments
on topological prevarieties and their generating sets and to relate our exposition
to an aspect of that of Hart and Kunen~\cite[Section~2.6]{HK99}.
Our presentation of the natural extension construction
 works with CT-prevarieties 
$\CA=\ISP(\CM)$, where 
usually $\CM$ has a single element though this is not essential.
The `home' of $n_{\CA}(\A)$, for each~$\A \in \CA$, is then the topological prevariety $\IScPnp(\CM_\Tp)$.   Our first comment is that $n_{\CA}(\A)$ is uniquely determined 
by~$\CM_\Tp$.
On the other hand, the universal
property characterising a Bohr compactification
for a general class of algebras,
$\CC$ say, involves \emph{all} members of $\CC^{\mathrm{Bt}}$ or
of $\CC^{\mathrm{ct}}$, as appropriate.  This---and knowledge of Pontryagin
duality and of the duality for semilattices---encourages Hart and Kunen
to introduce the notion of adequacy of a subclass
$\CK$ of a class $\CC_\Tp$ (of topological algebras):  this amounts to saying that the continuous homomorphisms from any element of~$\CC_\Tp$ into
structures in $\CK$ separate points. They do not, however, pursue this idea much further. We draw parallels here with the Boolean-topological version
 given by Jackson \cite[Lemma~2.2]{J08} of the classic Separation Theorem for
quasivarieties as recalled in \cite[Lemma~2.1]{J08}.  This separation result is elementary, but more significantly \cite{J08} throws light on the standardness problem from the perspective of topological residual bounds as compared to non-topological residual bounds and presents some
interesting examples in the context of IRF-prevarieties of finite type.

\section{Describing the natural extension and Bohr compactifications: the role of duality}
\label{nat-describe}

Our principal objective in this section is to demonstrate that  Bohr compactifications
can be concretely described for many classes $\CA=\ISP(\M)$ of algebras and, potentially,  of structures.   To this end we shall bring together 
two strands of theory.  The first of these strands comes from Section~\ref{sec:BviaN}. 
There  we  revealed that,  when $\CA$  is an IRF-prevariety,   strong coincidence of 
$b_0$ and $n_{\CA}$ is equivalent to standardness of the associated topological prevariety $\CA_\Tp$ (Theorem~\ref{thm:NatIsBohr}), and we gave an analogous result
when $\CA$ is  an infinitely generated  CT-prevariety (Theorem~\ref{thm:NatIsBohrInf}).  Our 
second strand of theory concentrates on the description of the natural 
extension.  We shall exploit duality theory to refine   the `brute force'   description
supplied by
Proposition~\ref{multi-pres:noduality}:    Theorem~\ref{thm:pairedadjunctions},
drawing on Theorem~\ref{multi-pres:single}, 
gives an amenable description of the natural extension
in case~$\M$ is a finite total structure which is dualisable.  Theorem~\ref{thm:b0-sumup}
  presents a catalogue of classes of algebras to which Theorems~\ref{thm:NatIsBohr} and~\ref{thm:pairedadjunctions} both apply, and for which thereby 
$b_0$ can be explicitly described.   The case when~$\M$ is infinite is more challenging, but
Proposition~\ref{prop:natViaDD} is noteworthy.  It   embraces all  cases in which ~$\M$, finite or infinite,  is self-dualising and, in combination with existing 
standardness results, sets in context known descriptions of 
the Bohr compactification~$b$  for abelian groups  (via Pontryagin duality) and 
of $b_0$ for 
semilatices (via Hofmann--Mislove--Stralka duality). In this section we are concerned with strong coincidence;  in Section~\ref{sec:egViaDuality} our examples 
will focus on weak coincidence.

 We preface our new results with   a very brief introduction to natural dualities for structures, 
as developed in~\cite{D06}, and follow this 
with  a broad brush summary of known dualisability results
for algebras.

We shall have two structures on the same set $M$ in play at the same time and it is convenient to adapt our notation to reflect  this.
Let $\M_1 = \langle M; G_1, H_1, R_1\rangle$ and $\M_2 = \langle M; G_2, H_2, R_2\rangle$ be structures on~$M$. 
Let $\M_2$ be
\emph{compatible with~$\M_1$}, meaning that each (partial) operation in $G_2$ ($H_2$) is a homomorphism (where defined) and each relation $r\in R_2$ as well as the domain of each partial operation $h\in
H_2$ form substructures of appropriate powers of $\M_1$,
and let $\Tp$ be a compact topology on $M$ that is compatible with $\M_1$.
Let $\MT_2$ be the corresponding \emph{alter ego} of~$\M_1$, that is, $\MT_2$ is the structure with topology $(\M_2)_\Tp$ obtained by adding the topology $\Tp$ to $\M_2$. Finally, let $\CA \coloneqq  \ISP(\M_1)$ and $\CX_\Tp \coloneqq \ISOcP(\MT_2)$ be respectively the prevariety of structures generated by $\M_1$ and the topological prevariety of structures with topology
generated by~$\MT_2$. In almost every case below, $M$ is finite, in which case $\CX_\Tp$ consists of Boolean-topological structures.

Note that we have switched here from $\IScPnp(\MT_2)$ to $\ISOcP(\MT_2)$. This is necessary as in general the dual of a one-element structure can be empty, for example when $\M_1$ is the two-element lattice with both bounds as nullary operations, and there might be no structure with a one-element dual, for example when $\M_1$ is the two-element lattice without nullary operations. If $\M_2$ has a total one-element substructure, then the $^+$ has no effect and we will use $\ISOcPnp(\MT_2)$ instead.

There are well-defined contravariant hom-functors $\mathrm{D} \colon \CA \to \CX_\Tp$ and $\mathrm{E} \colon \CX_\Tp \to \CA$
given on objects by
\[
\D \A \coloneqq \CA(\A, \M_1) \le (\MT_2)^A
\text{ and }\E \X \coloneqq \CX_\Tp(\X, \MT_2) \le \M_1^X,
\]
for all $\A \in \CA$ and all $\X \in \CX_\Tp$.
The construction of $\CA$ and $\CX_\Tp$ guarantees that the maps given by evaluation
\[
\esubA :\A \to \ED \A \text{ and } \epsub \X \colon \X \to \DE \X
\]
are embeddings.  Then $\langle \mathrm{D}, \mathrm{E}, e, \varepsilon\rangle$ is a dual adjunction between $\CA$ and~$\CX_\Tp$.
If, for all $\A\in\CA$, the map $\esubA$ is an isomorphism, then $\MT_2$ is said to yield a \defn{duality} on $\CA$  or we say that
$\MT_2$ yields a \defn{duality between} $\CA$ and $\CX_\Tp$.
Alternatively, we may say that $\MT_2$ \defn{dualises} $\M_1$. Also 
 $\MT_2$ yields a \defn{full duality} between $\CA$ and $\CX_\Tp$ if, in addition, for all $\X\in\CX_\Tp$, the map $\epsub \X$ is an isomorphism; then the functors $\mathrm{D}$ and $\mathrm{E}$ give a dual equivalence between the categories $\CA$ and $\CX_\Tp$.

The following theorem,
as it applies to an IRF-prevariety of algebras,
appears in \cite[Theorem~4.3]{nisp}.
The proof given in~\cite{nisp} extends immediately to total structures 
but not to structures in general.

\begin{theorem} \label{multi-pres:single}
Let $\M_1$ be a finite total structure.
Assume that $\M_2$ is a structure
compatible with $\M_1$ and that the topological structure
$\MT_2$
acts as a dualising alter ego for $\M_1$. Let
$\CA \coloneqq  \ISP{(\M_1)}$, let $\A$ belong to~$\CA$, and let $b \colon \CA(\A, \M_1) \to M$ be a map. Then
 $b$ belongs to $n_\CA(\A)$
if and only if
 $b$ preserves the structure on~$\M_2$.
\end{theorem}

When a duality for~$\CA$  is known, Theorem~\ref{multi-pres:single} describes the elements of $n_{\CA}(\A)$, which is defined topologically, in a way which is not overtly topological. But this description  is defective:  
$n_{\CA}(\A)$ is a topological structure and not merely a set.  Hence we seek a more categorical answer to the description problem.

Suppose we have compatible structures $\M_1$
and $\M_2$ on the same finite set $M$ and define $\CA=\ISP(\M_1)$,
$\CX_\Tp = \ISOcP(\MT_2)$ and
hom-functors~
$\mathrm{D}
\colon \CA \to \CX_\Tp$ and~
$\mathrm{E} \colon \CX_\Tp
 \to \CA$ as above.  We do not yet assume that we have a duality.
The compatibility relation between two structures
is symmetric, so that $\M_2$ compatible with $\M_1$ implies that $\M_1$ is compatible with $\M_2$. Thus we may swap the discrete
topology to the other side and repeat the construction using the alter ego $\MT_1$ of the structure $\M_2$ to obtain new categories
$\CA_\Tp \coloneqq  \IScPnp(\MT_1)$ of Boolean-topological structures and $\CX \coloneqq  \ISOPp(\M_2)$ of structures.
Now the contravariant hom-functors $\mathrm{F}\colon \CA_\Tp \to \CX$ and $\mathrm{G} \colon \CX \to \CA_\Tp$ are given by
\[
\F \A \coloneqq \CA_\Tp(\A, \MT_1) \le \M_2^A \text{ and } \G \X \coloneqq \CX(\X, \M_2) \le (\MT_1)^X. 
\]
We have maps given by evaluation $e_\A' \colon \A \to \GF \A$ and $\varepsilon_\X' \colon \X \to \FG \X$,
for all~$\A \in \CA_\Tp$ and all $\X \in \CX$, and  we
 obtain a new dual adjunction $\langle \mathrm{F}, \mathrm{G},
 e', \varepsilon'\rangle$ between $\CA_\Tp$ and~$\CX$.
Then we refer to $\langle \mathrm{D}, \mathrm{E}, e, \varepsilon\rangle$ and $\langle \mathrm{F}, \mathrm{G}, e', \varepsilon'\rangle$ as \defn{paired adjunctions} (see \cite[p.~587]{DHP12}).
 If $e_\A' \colon \A \to \GF \A$ is an isomorphism, for all~$\A \in \CA_\Tp$, then we say that \emph{$\M_2$ yields a
 duality between $\CA_\Tp$ and $\CX$}. If, in addition, $\varepsilon_\X' \colon \X \to \FG \X$ is an isomorphism,
 for all $\X \in \CX$, then we say that \emph{$\M_2$ yields a full duality between $\CA_\Tp$ and $\CX$}.

\begin{figure} [ht]
\begin{tikzpicture}
 [node distance=1.8cm,
 auto,
 text depth=0.25ex,
move up/.style=   {transform canvas={yshift=1.5pt}},
move down/.style= {transform canvas={yshift=-1.5pt}},
move left/.style= {transform canvas={xshift=-1.5pt}},
move right/.style={transform canvas={xshift=1.5pt}}]
\node (X) {$\CA$};
\node (YT) [right=of X] {$\CX_\Tp$};
\node (XT) [below=of X] {$\CA_\Tp$};
\node (Y) [right=of XT] {$\CX$};
\draw[-latex,move up]   (X) to node {$\mathrm{D}$} (YT);
 \draw[latex-,move down] (X) to node [swap] {$\mathrm{E}$} (YT);
\draw[-latex,move left]           (X) to node [swap] {$n_{\CA}$} (XT);
\draw[-latex,move right]           (XT) to node [swap]  {$^\flat$} (X);
\draw[-latex,move up]   (XT) to node {$\mathrm{F}$} (Y);
\draw[latex-,move down] (XT) to node [swap] {$\mathrm{G}$} (Y);
\draw[-latex,move left]        (Y) to node   {$n_{\CX}$} (YT);
\draw[-latex,move right] (YT) to node {$^\flat$}   (Y);
\end{tikzpicture}
\caption{Paired adjunctions\label{fig:paired-natext}}
\end{figure}

The following theorem  extracts   from
 \cite[Theorem 2.3]{DHP12} only the assertions  we shall need.
 The generalisation from algebras to total structures is completely straightforward.

\begin{theorem}
\label{thm:pairedadjunctions}
Let $\M_1$ be a finite total structure, let $\M_2$ be a structure
compatible with~$\M_1$ and define $\CA$ and $\CX_\Tp$ as above.
Of  the  following conditions, \upsh{(2)} and \upsh{(3)} are equivalent
and implied by \upsh{(1)}.
\begin{newlist}
\item[\upshape (1)]
$\MT_2$ yields a duality between $\CA$ and $\CX_\Tp$;
\item[\upshape (2)]
the outer square
of Figure~\upsh{\ref{fig:paired-natext}}
commutes, that is, $n_\CA(\A) =
 \G{\D\A^\flat}$, for all $\A\in \CA$;
\item[\upshape (3)]
$n_\CA(\A)$ consists of all maps $\alpha \colon \CA(\A, \M_1)
 \to M$ that preserve the structure on $\M_2$, for all $\A\in \CA$.
\end{newlist}
\end{theorem}

We now give our promised summary of dualisability results,  concentrating
on algebras rather than total structures more widely.
We have to acknowledge that a number of important varieties of algebras are 
not CT-prevarieties and  cannot be brought  within the scope of natural 
duality theory. As Hart and Kunen show, Bohr compactifications, abstractly defined, will exist for such varieties, but
neither they, nor we, have machinery to access such compactifications concretely.
The classes of lattices, semigroups,  and rings, in particular, fall under this heading.
However  all these classes have interesting subvarieties which 
we may profitably consider.

We can draw on a very extensive literature for examples of finitely generated
prevarieties for which explicit dualities are known and for confirmation that others fail to have a natural duality (see \cite{NDftWA, PD05} and the references therein). For the benefit of those unfamiliar with this literature we give the briefest possible summary.  We initially consider a quasivariety $\CA=\ISP(\M)$, where $\M$ is a finite algebra.

\begin{named}[Two-element algebras]
Taking $\M$ to be the  $2$-element algebra in the following
varieties (where $\HSP(\M) = \ISP(\M)$),   we encompass important classic dualities:
\begin{newitemize}
\item \textbf{Stone duality} for Boolean algebras;
\item \textbf{Priestley duality} for distributive lattices, with or without
bounds, depending on whether
bounds of $\M$ are included in the signature as nullary
operations; 
\item  \textbf{Hofmann--Mislove--Stralka duality} for (meet) semilattices, with or without bounds.
\end{newitemize}
  Not all $2$-element algebras are dualisable; the implication algebra $\langle \{0,1\}; \to\rangle $ provides a classic example.
The case  $|\M|= 2$  is fully analysed in \cite[Section~10.7]{NDftWA}.
\end{named}

\begin{named}[Lattices and lattice-based algebras]
Assume that $\M$ is a lattice or has a lattice reduct.  Then $\M$ is dualisable and the alter ego
can be taken to be purely relational, with relations of arity no greater than~$2$.
The situation in which  $\M$ has a (bounded) distributive lattice reduct has been
 thoroughly researched: very amenable dualities have been found  for many
 familiar varieties, assisted by the theory of optimal dualities and by the piggyback method.
\end{named}

\begin{named}[Semilattice-based algebras]
In contrast to the lattice-based case,
a semilattice-based algebra may or may not be dualisable (see
 \cite{DJPT,CDPR}).
 \end{named}

\begin{named}[Groups and semigroups]
Modulo an unpublished proof, the dualisable groups have been classified.
There is only
fragmentary information on dualisability of semigroups, which form a large and diverse class, with only certain subclasses (for example bands) analysed in depth. (See ~\cite{J15} for a detailed discussion of dualisability for both groups and semigroups.)  
\end{named}

\begin{named}[Commutative rings]
Here we note the characterisation  by Clark {\it et al.}~\cite{CISSW} of those  finite commutative rings
which are dualisable and of the amenable dualities available in some particular cases.
\end{named}

\begin{named}[Unary algebras]
Such algebras exhibit very varied behaviour.
Particularly for small $|M|$, they have been comprehensively studied, most notably in \cite{PD05}, as path\-finder examples for general theory.
\end{named}

A miscellany of sporadic examples of dualisable finite algebras could be added to the above list.  
Examples of dualisable finite structures which are not algebras can be found in \cite{D06,DHP12,pig,J10}.  Our focus in this section is on
the finitely generated case, but for completeness we draw attention to our recent paper \cite{pig}
which provides
theory embracing the possibility of an infinite generator.

Let us now pull together  threads  from this section and the previous one to
present a theorem on zero-dimensional Bohr compactifications.

\begin{theorem} \label{thm:b0-sumup}
Let  $\CA = \ISP(\M_1)$, where $\M_1$ is a finite algebra,
with associated topological
prevariety $\CA_\Tp =
\IScPnp(\MT_1)$. Assume that $\M_1$ is a lattice, or a dualisable semigroup, group,  ring, or unary algebra,
or any other dualisable algebra,
and assume that
$\CA_\Tp$ is standard.
Let $\MT_2$ be a dualising alter ego 
of $\M_1$.
Then
\begin{newlist}
\item[\upshape (i)]
 $b_0(\A) = n_{\CA}(\A)$ for each $\A\in \CA$, and hence
\item[\upshape (ii)] $b_0(\A)$
is given by Theorem~{\upshape\ref{thm:pairedadjunctions}}. 
  \end{newlist}
\end{theorem}

A small number of  familiar  examples 
fit into the scheme  envisaged in Theorem~\ref{thm:pairedadjunctions} in a rather special way.
Assume that we have a finite \emph{self-dualising}  structure $\M$, that is, $\M_\Tp$
acts as a dualising alter ego for the prevariety $\CA= \ISP(\M)$. In this situation we have a natural extension which has a particularly simple, indeed perhaps deceptively simple, description. By Theorem~\ref{thm:pairedadjunctions}, the natural extension $n_{\CA}(\A)$ is then
just $\mathrm{D}(\mathrm{D}(\A)^\flat)$. Significantly,  this description of the natural extension via the iterated duality functor applies even if the self-dualising algebra is infinite; but this requires a different argument (cf.~\cite[p.~36]{Ho64}).

\begin{proposition}\label{prop:natViaDD}
Let $\M$ be a structure \textup(finite or infinite\textup) and define $\CA\coloneqq \ISP(\M)$. Assume that $\Tp$ is a compatible compact
topology on $\M$ such that $\M_\Tp$ fully dualises~$\M$. Then, for each $\A\in \CA$, the natural extension $n_{\CA}(\A)$ of $\A$ is isomorphic to $\D {\D \A^\flat}$.
\end{proposition}

\begin{proof}
Let  $\A \in \CA$. The universal property (cf.~Theorem~\ref{lem:4}) implies that every homomorphism $g \colon \A \to \M$ has a unique lifting to a continuous homomorphism
$h\colon n_{\CA}(\A) \to \M_\Tp$ such that $h \circ e_{\A} = g$.
Since $\M$ is self dualising,
it follows easily that $\D\A^\flat$ and
$\E {n_{\CA}(\A)}$ are isomorphic as structures.
As $\M$ is fully self-dualising, we conclude that
$n_{\CA}(\A ) \cong \DE {n_{\CA}(\A)} \cong \D{\D \A^\flat}$.
\end{proof}

The following varieties 
are covered by Proposition~\ref{prop:natViaDD}.

\begin{newitemize}
\item \textbf{Meet semilattices
with~$1$}.  In this case, Hofmann--Mislove--Stralka
duality \cite{HMS74} (or see \cite[4.4.7]{NDftWA}) applies: here we have dual equivalences between
  \begin{alignat*}{2}
    \CS &= \ISP(\twoBf)  &\quad& \text{[$\wedge$-semilattices  with $1$]}, \\
\CS_\Tp
&=  \IScPnp(\twoBf_\Tp)  && \text{[compact zero-dimensional $\wedge$-semilattices with $1$]},
\end{alignat*}
where $\twoBf = \langle \{ 0,1\}; \wedge, 1 \rangle$. It is easy to see that,
for $\mathbf{S} \in \CS$, the natural extension $n_{\CS}(\mathbf{S})$ can be identified with
the repeated filter lattice $\Filt(\Filt(\mathbf{S}))$, equipped with the unique topology making it a member of $\CS_\Tp$, {\it viz.}~the Lawson topology.
 Discussion of the natural extension from this perspective is given in
\cite{GP-s1}.
  \item \textbf{Abelian groups of exponent $n$}
\cite[Theorem~4.4.2]{NDftWA}.
Here we have a full duality between the categories $\CA^n = \ISP(\mathbb Z_n)$ of abelian groups of exponent~$n$ and $\CA^n_\Tp =  \IScPnp(\mathbb
{Z}_n^\Tp)$ of Boolean topological abelian groups of exponent $n$.

\item \textbf{Abelian groups}. Pontryagin's famous dual equivalence between the categories $\CA = \ISP(\mathbb T)$ of abelian groups and $\CA_\Tp =  \IScPnp(\mathbb T_\Tp)$ of compact topological abelian groups was brought within the scope of natural duality theory from the beginning
\cite[Theorem 4.1.1]{D78}. 
Here $\mathbb T$ is the circle group and 
    $\Tp$ is the Euclidean topology.

    \item \textbf{Semilattices with automorphism}. The variety we considered in
Example~\ref{jezek}
is self-dualising \cite[Theorem~8.2(3)]{DJPT}.
\end{newitemize}

In each of these examples, the dual category $\IScPnp(\M_\Tp)$ is standard,
or compact-standard (in the case of abelian groups),
or Boolean standard (in the case of semilattices with automorphism).
Thus, in each case we can combine
Theorem~\ref{thm:NatIsBohr} with Proposition~\ref{prop:natViaDD} to conclude that the Bohr compactification or zero-dimensional Bohr compactification of $\A$ is isomorphic to $\D{\D \A^\flat}$. This description is well known in the case of abelian groups (see \cite{Ho64})
and the case of meet semilattices with 1 (see~\cite[Theorem I-3.10 and Definition~I-3.11]{HMS74}).

Further examples of the same type are:  Boolean groups; meet semilattices with~$0$ and join-semilattices with~$0$ or with~$1$ (see \cite[Table~10.2]{NDftWA}); certain semilattice-based algebras (see the Semilattice-Based Self-Duality Theorem 7.4 in~\cite{DJPT}); and other self-dualising situations in which the machinery of \cite[Section~2]{pig} applies.

In traditional duality theory, one often encounters dual equivalences between categories $\CA$ and $\CX_\Tp$ where one of $\CA$ and $\CX$ is a category of algebras and the other is a category of
structures  which are often purely relational.
 We shall focus on Bohr compactifications of purely relational structures in the next section.  Here we wish to highlight
with some examples the way in which  both operations and relations can
arise on each side in a duality.
We shall consider  ordered (but not lattice-ordered) algebras $\M_1$ such that the dualising structure $\MT_2$ is not an algebra. The theorem below providing examples of such dualities comes, as usual,  by observing that a known theorem for algebras extends to total structures. The only proof that is required is an instruction to read the old proof and note that it still works. (One needs to know that the Preservation Lemma~\cite[1.4.4]{PD05} still holds, which it does provided $\M_1$ has no partial operations.)
The result for algebras can be found in \cite[2.1.1]{PD05}.

\begin{theorem}  
\label{thm:latticeops}
Let $\M_1 = \langle M; F, R\rangle$ be a finite total structure that has binary homomorphisms $\vee$ and $\wedge$ such that $\langle M;\vee, \wedge\rangle$ is a lattice. Then $\MT_2 \coloneqq  \langle M;\vee, \wedge, R_{2|M|}, \Tp \rangle$ yields a duality on $\ISP(\M_1)$, where $R_{2|M|}$ is the set of $2|M|$-ary relations compatible with~$\M_1$.
\end{theorem}

As an immediate corollary we get the following nice result.

\begin{theorem} [Lattice Endomorphism Theorem~{\cite[2.1.2]{PD05}}]\label{thm:latendos} 
Endomorphisms and compatible orders of finite lattices yield dualisable ordered unary algebras. More precisely, let $\M = \langle M;\vee, \wedge\rangle$ be a finite lattice, let $F\subseteq \Endo(\M)$ and let 
$\leq$ be an order on $M$ that is preserved by both $\vee$ and~$\wedge$. Then $\MT_2 \coloneqq  \langle M;\vee, \wedge, R_{2|M|}, \Tp \rangle$ dualises $\M_1 \coloneqq \langle M; F, \leq \rangle$, where $R_{2|M|}$ is the set of $2|M|$-ary relations compatible with~$\M_1$.
\end{theorem}

\begin{example}
Let $\M_1 = \langle \{0, 1, 2\}; u, d, \leq\rangle$ be
a unary algebra with $u(0) = 1$, $u(1) = u(2) = 2$ and $d(2) = 1$, $d(1) = d(0) = 0$, enriched with either the usual order $0 < 1 < 2$, the uncertainty order $0 < 1 > 2$ of Kleene algebra duality fame or the order whose only proper comparability is $1 < 2$ of Stone algebra duality fame. Since each of these orders is compatible with the $\vee$ and $\wedge$ of the three-element lattice, the Lattice Endomorphism Theorem~\ref{thm:latendos} tells us that the alter ego $\MT_2 \coloneqq  \langle M;\vee, \wedge, R_6, \Tp \rangle$ yields a duality on $\CA = \ISP(\M_1)$.
By Theorem~\ref{multi-pres:single}, the natural extension $n_{\CA} (\A)$ of an ordered algebra
$\A\in \CA$ is simply described as the algebra consisting of all $\{\vee,\wedge\}\cup {R_6}$-preserving maps from $\CA(\A, \M_1)$ to $\M_2$. We have not investigated whether the natural extension will provide a concrete realisation of the zero-dimensional Bohr compactification in this case.
\end{example}

\section{Natural extensions and Bohr compactifications:  making  use of topology-swapping}
\label{sec:egViaDuality}

We consider once again the scenario presented in
Figure~\ref{fig:paired-natext}, retaining the notation from Section~\ref{nat-describe}.
 So assume
that $\M_1 = \langle M; G_1, R_1\rangle$ and $\M_2 = \langle M;  G_2,H_2, R_2\rangle$
are compatible
structures on the finite set $M$, with $\M_1$ total.
We define
\begin{alignat*}{2}
\CA &= \ISP(\M_1), \qquad \qquad & \CX_\Tp &=
\ISOcP(\MT_2), \\
\CA_\Tp
&= \IScPnp(\MT_1),
& \CX &= \ISOP(\M_2). 
\end{alignat*}
We set up the
hom-functors
$\mathrm{D}$, $\mathrm{E}$, $\mathrm{F}$ and $\mathrm{G}$ as before.
Thus we envisage trying to swap the topology from $\M_2$ to $\M_1$.
We seek   conditions under which we can upgrade the statement of  Theorem~\ref{thm:pairedadjunctions} so as to assert that both adjunctions are
dual equivalences. 
  When this occurs we shall say we have \emph{paired  full dualities}.

We shall highlight two theorems which yield paired full dualities.  The first is the TopSwap Theorem for (total) structures.  This was obtained for algebras in \cite[Theorem~2.4]{DHP12}.
We preface its statement with a technical observation.
We follow~\cite{DHP12} in indicating that
it is only necessary that we have
a duality, or full duality, 
at the finite level (that is, on the full subcategory of~$\CA$ consisting of the finite objects).
For the theorem as we shall apply it, we do not make use of the weakened assumptions.  But it would be disingenuous to exclude them:
the core of the proof  in \cite{DHP12},
which applies equally well to total structures, 
 concerns
what happens at the finite level, with the lifting to the whole class
relying solely on categorical generalities.

\begin{TopSwap}\label{thm:paireddualities}
Let $\M_1$ be a finite total structure of finite type, let $\M_2$ be a structure compatible
with~$\M_1$ and define the categories $\CA$, $\CA_\Tp$, $\CX$ and $\CX_\Tp$
as above. 
\begin{newlist}

\item[{\upshape (1)}]
If\/
$\MT_2$
yields a finite-level
duality between $\CA$
and $\CX_\Tp$, then $\M_2$ yields a duality between $\CA_\Tp$ and $\CX$.

\item[{\upshape (2)}]
 If\/ $\MT_2$
 yields a finite-level
 full
duality between $\CA$ and $\CX_\Tp$, then
the adjunction 
$\langle \mathrm{F}, \mathrm{G}, e', \varepsilon'\rangle$ is a dual equivalence between the categories $\CA_\Tp$ and~$\CX$.
\end{newlist}
\end{TopSwap}

Combining the TopSwap Theorem for Structures with
Theorem~\ref{thm:pairedadjunctions}
we obtain natural extensions in partnership in the manner described in the next result. \begin{corollary}\label{descr_of_compatifications}
Let $\M_1$ be a finite total structure of finite type, let $\M_2$ be a
total structure compatible
with~$\M_1$ and define the categories $\CA$, $\CA_\Tp$, $\CX$ and $\CX_\Tp$
as above. Assume that \/ $\MT_2$ yields a full duality
between $\CA$ and $\CX_\Tp$. Then $\M_2$ yields a full duality between
$\CA_\Tp$ and $\CX$ and
\begin{newlist}

\item[{\upshape (1)}]
$\begin{aligned}[t]
n_{\CA}(\A) &= \G{\D{\A}^\flat}, \text{ for all }\A\in \CA \text{ and }\\
 n_{\CX}(\X) &= \D{\G{\X}^\flat}, \text{ for all  }\X\in \CX,
\end{aligned}$\\[.8ex] 
so that Figure~{\upshape\ref{fig:paired-natext}} combines two commuting squares;  
\item[\upshape (2)]
$n_{\CA}(\A)$ consists of all maps $\CA(\A, \M_1) \to M$ that preserve the structure on~$\M_2$, for all $\A\in \CA$;

\item[\upshape (3)]
$n_{\CX}(\X)$ consists of all maps $\CX(\X, \M_2) \to M$ that preserve the structure on~$\M_1$, for all $\X\in \CX$.

\end{newlist}
\end{corollary}

We warn that the requirement that $\M_2$ be a total structure means that topology-swapping cannot be applied to obtain paired natural extensions
 in circumstances where partial operations have to be included
in a  dualising alter ego $\MT_2$ for $\M_1$
 in order to upgrade a duality to a full (in fact, a strong) duality (see
\cite[Chapter~3]{NDftWA}), 
or where partial operations are present
in a dualising alter ego (as happens, for example for dualisable commutative
rings \cite{CISSW}
and for dualisable non-abelian groups~\cite{QS}).

We have seen that the natural extension  provides  a common framework
for a range of universal constructions on algebras and purely
relational structures, so indicating that these do not relate to  quite different worlds.
 But Corollary~\ref{descr_of_compatifications} gives us  more.
Not only does $\M_2$ yield a description of natural extensions in~$\CA$, but it also yields a duality on the category $\CA_\Tp$ within which these natural extensions live. Moreover a corresponding statement also holds for~$\M_1$ and~$\CX$.

We now  present two examples of paired full dualities between categories of total structures. These examples show that certain famous compactifications arise as paired natural extensions and that useful information stems from  the linkage. Here the natural extension functor is, or  generalises,
a Stone--\v {C}ech compactification functor. In this setting,
the functor is, or may be, \emph{defined} as in
Definition
\ref{def:natext} above.

\begin{example}[Boolean algebras and the Stone--\v {C}ech compactification]
\label{ex:Bool}

Here we have a well-known classic. It arises from Stone duality between Boolean algebras and Boolean spaces and its topology-swapped counterpart, the duality between  sets and  Boolean-topological Boolean algebras.
The categories involved are
 \begin{alignat*}{4}
    \CB &= \ISP(\twoBf)  &~~&\text{[Boolean algebras]},  &\qquad
 \CZ_\Tp &= \ISOcPnp(\twoT_\Tp) &~~&\text{[Boolean spaces]},\\
\CB_\Tp &= \ISP(\twoBf_\Tp) &&\text{[Boolean topological BAs]}, &
\CZ &= \ISOcPnp(\twoT) &&\text{[Sets]},
\end{alignat*}
where the generating objects are the  unique two-element structures, with or without topology as appropriate, in the categories concerned. Theorem~\ref{thm:pairedadjunctions} asserts that
the functor $n_{\CB}$ sends a Boolean algebra to the powerset of its dual space. For any set~$Z$, the Stone-\v {C}ech compactification $\beta Z$
is well known to be zero-dimensional, and  so is a member of
$\CZ_\Tp$. Therefore by Proposition~\ref{prop:coincide} the Stone-\v {C}ech compactification {\it alias} the (zero-dimensional) Bohr compactification
coincides with the natural extension. Thus, we have
\[
\beta Z = b(Z) = b_0(Z) = n_{\CZ}(Z).
\]

Corollary~\ref{descr_of_compatifications}
now tells us that $\beta Z$
 is the dual space of the powerset Boolean
algebra $\powerset Z$. Thus we may, should we so
choose,  regard our construction here as a way to obtain the dual space of a powerset algebra.  The relationship between the dual of a powerset algebra
and the Stone--\v {C}ech compactification is of course very well known
and can be obtained by a variety of methods different from ours; see for example
\cite[Section~8.3]{HBA}, \cite[pp.~230--232]{Eng}
or 
\cite[16.2.5]{Sem}.
\end{example}

We now consider the natural extension on the category $\CP$ of ordered
sets.
We may regard $n_{\CP}$  as
a compactification functor on~$\CP$, paralleling that of the Stone--\v{C}ech compactification functor on sets.
Compactifications of ordered sets, and more generally of ordered topological spaces, has been extensively studied, following the introduction by Nachbin of the order-compactification which now bears his name \cite{N76}.
This construction provides  a reflection of $\CP$ into the category of compact ordered spaces.
Let $\PP=\langle Y; \leq\rangle$ belong to $\CP$.
Then the topological structure 
$ \langle Y; \leq,\Tp\rangle$
  is \emph{compact ordered} (an alternative term is
\emph{compact order-Hausdorff}) if $\Tp$ is compact and $\leq$ is closed in 
$Y \times Y$,
from which it follows that the topology is Hausdorff. Thus
the Bohr compactification
$b(\PP)$ coincides with the Nachbin
order-compactification
of~$\PP$, which we shall denote by
$\Nach \PP$.
But more is true.  In \cite{BeMo}, Bezhanishvili and Morandi study what they call \emph{Priestley order-compactifications} for a suitable class of ordered spaces,
which includes those that are discretely topologised.
Crucially for our purposes
they
 demonstrate that
$\Nach \PP$ is a Priestley space for any $\PP \in \CP$
\cite[Corollary~4.7]{BeMo}; this result was
 also proved, by a different method, by Nailana \cite{Nai}.
As a consequence, Proposition~\ref{prop:coincide} now provides the following 
noteworthy result.
The topological prevariety $\IScP(\twoT) $
of Priestley spaces is non-standard; recall Example~\ref{ex:Stralka}.
Nevertheless
weak coincidence does occur.

\begin{proposition}  \label{NachisNat}
For any ordered set~$\PP$,
\[
\Nach \PP = b(\PP)
= b_0(\PP) = n_{\CP}(\PP).
\]
\end{proposition}

We now discuss the paired full dualities between $\CP_\Tp$ (Priestley spaces)
and $\CCD$ (bounded distributive lattices) and the ramifications of this partnership.

\begin{example}[Bounded distributive lattices and  the  Nachbin
order-compactification]\label{ex:ord comp}

Here we build on the partnership between
Priestley duality for bounded distributive lattices and the duality, due to Banaschewski~\cite{Ban76},
between ordered sets and Boolean-topological bounded distributive lattices.
Accounts of this partnership are given in 
\cite[Example~4.1]{DHP12} and \cite[Section~4]{pig}. 

The categories involved are
\begin{alignat*}{4}
    \CCD  &= \ISP(\twoBf)    &~~&\text{[bounded DLs]},   &\qquad \CP_\Tp & = \ISOcPnp(\twoT_\Tp)       &~~&\text{[Priestley spaces]},\\ 
\CCD_\Tp &=\IScPnp(\twoBf_\Tp) &&\text{[Bt bounded DLs]},  &
\CP &= \ISOP(\twoT) &&\text{[ordered sets]};
\end{alignat*}
once again the generators in the four cases are given by two-element objects in the categories concerned.
Corollary~\ref{descr_of_compatifications}
tells us
that the natural extension $n_{\CCD} (\Lalg)$ of a bounded distributive lattice $\Lalg$ is the Boolean-topological lattice consisting of all order-preserving maps from $\CCD(\Lalg,\twoBf)$ to
$\twoT$. Moreover,  for any  $\PP\in \CP$
the set of elements of $n_{\CP}(\PP)$
consists of all lattice homomorphisms from
$\CP(\PP,\twoT)$
to $\twoBf$.

There is more to be said about the Nachbin order-compactification, {\it alias} Bohr compactification, in relation to duality.
Drawing on Corollary~\ref{descr_of_compatifications},
we see that
for any $\PP \in \CP$  we have
 $n_{\CP}(\PP) =   \mathrm D(\mathrm G (\PP)^\flat)$
and for any
$\Lalg\in \CCD$ we have $n_{\CCD}(\Lalg) = \mathrm G(\mathrm D(\Lalg)^\flat)$.
The first of these gives immediately that, for any ordered set~$\PP$,  the Priestley dual space of an up-set
lattice $\mathcal U(\PP)$ is the Nachbin order  compactification  $\Nach \PP$.

 It is well known, and has been proved in various  ways (see
\cite{BGMM,nisp,DP12}),  that for a
bounded distributive lattice $\Lalg$ we have
$n_{\CCD}(\Lalg) \cong \pro_{\CCD}(\Lalg) \cong \Lalg^\delta$, where $\pro_{\CCD}(\Lalg)$ and $\Lalg^\delta$ denote respectively the profinite completion and the canonical extension of~$\Lalg$.   It follows that the Priestley dual space of $\Lalg^\delta$ is
$\Nach{\mathrm D(\Lalg)^\flat}$.
This recaptures
\cite[Corollary 5.4]{BeMo} (see also \cite[Proposition 3.4]{BGMM}). Our proof is different from that in \cite{BeMo} and separates the component parts of the argument in a transparent way.
\end{example}

The piggyback technique is a time-honoured way to find amenable natural dualities for prevarieties of algebras having reducts in dualisable prevarieties, in particular, in $\CCD$ or $\CS$  (see \cite[Chapter~7]{NDftWA}).
This technique has been extended and refined in~\cite{pig}
so as to apply to CT-prevarieties.  Moreover under a variety of conditions, albeit stringent, the piggyback technique can be used directly to yield
paired full dualities: see the 
Two-for-one Piggyback Strong Duality Theorem, \cite[Theorem~3.8]{pig}.   This result differs from
Theorem~\ref{thm:paireddualities} in several respects.
As the name implies, it produces paired dualities
both of which are strong (strongness, as opposed to fullness, is not relevant in  this paper but  is important in other contexts). The theorem  is not {\it a priori}
restricted to the finitely generated case, and even there
it bypasses the finite type restriction of the TopSwap Theorem.
  When applicable, the specialisation to $\CCD$-based
prevarieties,
\cite[Theorem~4.5]{pig},
yields paired dualities very tightly tied to the paired dualities
for the base categories  $\CCD$ and~$\CP$
discussed in Example~\ref{ex:ord comp}.  We mention one particular
application, to the variety $\CO$ of Ockham algebras
\cite[Theorem~4.6]{pig}. 
Here there are mutually compatible structures on the infinite set 
$M=\{0,1\}^{\mathbb N_0}$, with  $\M_1$ generating~$\CO$ and $\M_2$ carrying an order and an operation
which is order-reversing with respect to the order.

Over more than 30~years  the variety of
Ockham algebras and its  finitely generated subvarieties have provided a rich source of  examples which have been influential in driving forward
the general theory of natural dualities.  Below
we consider paired full dualities for certain very special  prevarieties of~$\CO$ (in fact they are varieties, as we confirm in Remark~\ref{rem:ock}).
Moving to the dual side
 we  shall identify  an infinite family  of
classes $\CX$ which exhibit the same  behaviour for compactifications as does our hitherto isolated example $\CP$: 
 weak coincidence of $n_{\CX}$, $b_0$ and~$b$,
with, 
Stralka-fashion, the associated topological prevariety non-standard.
Our purpose is to demonstrate that such behaviour is not a rare phenomenon rather than to investigate Ockham algebra varieties {\it per se}
and we shall accordingly not attempt to make our account self-contained,
referring any interested reader to \cite{NDftWA,DNP}, 
and references therein,  for background.   In particular we shall make use of the well-known restricted Priestley duality for Ockham algebras and its subvarieties; see for example \cite{Gold81} and \cite[Section~7.4]{NDftWA}.
The varieties we shall consider are covered by \cite[Theorem 4.6]{pig}
but we shall sidestep this.
The dualities can be found in~\cite[Section~3]{DP87} or in~\cite{DNP}; 
the latter provides the axiomatisations of the dual categories which we shall crucially need.

In preparation for Theorem~\ref{thm:pigcoincide} we present some facts about prevarieties and topological prevarieties  whose objects have reducts in~$\CP$ and~$\CP_\Tp$, respectively.

\begin{lemma}  \label{lem:Nach-props} 

\begin{newlist}
 \item[\upshape(i)]
Let $\PP$ be an ordered set.  Then $\Nach {\PP^\partial}= \Nach \PP^\partial$,
where $\PP^\partial$ denotes the order-theoretic dual of~$\PP$.
\item[\upshape(ii)] 
Let $\PP_1,\PP_2$ be ordered sets and $f \colon \PP_1 \to \PP_2$ be a map which
is order-preserving \textup(respectively, order-reversing\textup).  Then $f$ 
has an extension to a map
$\overline{f} \colon \Nach {\PP_1} \to \Nach {\PP_2}$ which is continuous and order-preserving
\textup(respectively,
order-reversing\textup).
\end{newlist} 
\end{lemma}

\begin{proof}
Consider (i).
Let $\PP$ be an ordered set.  From above, $\Nach \PP$ is the Priestley dual space of the up-set lattice $\mathcal U(\PP)$.  Likewise, $\Nach {\PP^\partial}$ is the dual space of $\mathcal U(\PP^\partial)$, which is order 
anti-isomorphic 
to $\mathcal U(\PP)$, via the complementation map.  Now use the well-known fact about Priestley duality that, 
for any $\Lalg \in \CCD$, we have $\mathrm D (\Lalg^\partial)$ 
homeomorphic and order anti-isomorphic to $\mathrm D(\Lalg)$.  
Putting all this together we obtain
$(\Nach {\PP^\partial})^\partial = \Nach \PP$  (up to order homeomorphism
in $\CP_\Tp$). By uniqueness, and then  flipping the order,  we deduce that
$\Nach {\PP^\partial} = (\Nach \PP)^\partial$.

 The proof of (ii) is
 an almost immediate consequence of the fact that $\beta_\leq$ is a functor, with use being made of (i) when
$f$ is order-reversing.
\end{proof}

In the proof of the following lemma the claim concerning the lifting of atomic formulas holds quite generally for total structures. For the application we shall make of the lemma, the proof of the claim is entirely elementary as the atomic formulas have a very simple form: $g(x_1) \leq h(x_2)$ or $g(x_1) = h(x_2)$, where $x_1, x_2 \in \{x, y\}$ with $g$ and $h$ in the monoid of self maps of $M$ generated by $F$.

\begin{lemma}\label{lem:Cornish}
Let $\M =\langle M; F, \leq\rangle$ where $\langle M; \leq\rangle$ is a finite ordered set that is not an antichain and $F$ is a set of unary operations on $M$ each of which is order-preserving or order-reversing, and define $\CX \coloneqq \ISOP(\M)$ and $\CX_\Tp\coloneqq \ISOcPnp(\M_\Tp)$. Assume that there is a set $\Sigma$ of atomic formulas in the language of $\M$ such that a structure with topology $\X = \langle X; F, \leq, \Tp\rangle$, with the same signature as $\M$, belongs to $\CX_\Tp$ if and only if
\begin{newlist}
 \item[\upshape(i)]
$\langle X; \leq, \Tp\rangle$ is a Priestley space,

\item[\upshape(ii)] 
for all $f\in F$, if $f$ is order-preserving \textup(respectively, order-reversing\textup) on~$\M$, then $f$ is continuous and order-preserving \textup(respectively, 
order-reversing\textup) on~$\X$,

\item[\upshape(iii)] 
$\X$ satisfies $\sigma$, for all $\sigma\in \Sigma$.

\end{newlist} 
Then $b(\X) = b_0(\X) = n_{\CX}(\X)$, for all $\X\in \CX$.

\end{lemma}

\begin{proof}
Let $\X \in \CX$ and form the Nachbin order compactification
$\Y\coloneqq  \Nach {\X^{\triangledown}}$, where $^\triangledown$ forgets the
maps in $F$.  This is a Priestley space.   In addition, by Lemma~\ref{lem:Nach-props}, each order-preserving (respectively, order-reversing) map $f \in F$ on $\X$ has a unique lifting to a continuous order-preserving (respectively, order-reversing) map, which we  also denote by~$f$, on $\Y$.  It is an easy exercise to show that 
any atomic formula holding on $\X$  
also holds on $\Y$ (since $\iota_{\X^\triangledown}(\X^\triangledown)$ is dense in $\Y$, each $g\in F$ is continuous and $\leq$ is closed).
We have shown that $\Y$ enriched with $f$, for each $f\in F$, satisfies (i)--(iii) and so belongs to $\CX_\Tp$, by our assumptions. We claim that this topological structure serves as $b(\X)$.

Let $\Z \in \CX^{\text{ct}}$ and take a homomorphism $g\colon \X \to \Z^\flat$.
Then $g^\triangledown \colon \X^\triangledown \to 
(\Z^\flat)^\triangledown$ lifts uniquely to a Priestley space morphism $h$ from $\Nach {\X^\triangledown} = b(\X^\triangledown)$ 
to $\Z^\triangledown$. Since $h$ commutes with each $f\in F$ when restricted to a dense subset, continuity guarantees that $h$ commutes with each $f\in F$ on $\Nach {\X^\triangledown}$. This yields the universal property demanded of~$b(\X)$.
By Proposition~\ref{prop:coincide}(ii), since
$b(\X)\in \CX_\Tp$ and $\CX_\Tp\subseteq \CX^{\mathrm{Bt}}$, we conclude that 
$b(\X) = b_0(\X) = n_{\CX}(\X)$.
\end{proof}

The topological prevarieties
$\CX_\Tp$ we shall consider in Theorem~\ref{thm:pigcoincide} are non-standard.
This follows from a very general, but unpublished, result concerning ordered unary algebras~\cite{BCDP}.  Therefore Theorem~\ref{thm:NatIsBohr} does not apply and hence we
need instead to exploit  Lemma~\ref{lem:Cornish} in order to  prove  our theorem.

\begin{theorem} \label{thm:pigcoincide}
There exists a countably infinite family $\mathcal  M$ of finite
pairwise non-isomorphic subdirectly irreducible Ockham algebras
such that each $\M_1 \in \mathcal  M$ has  the following properties:
\begin{newlist}
\item[\upshape (1)]  there exists a total structure $\M_2 = \langle M; u, \preccurlyeq\rangle $, with $\preccurlyeq$ an order on~$M$  and $u\colon M \to M$ an order-reversing map, such that $\M_2$ is compatible with $\M_1$ and 
    $\MT_2$ yields a full duality between $\CA \coloneqq \ISP(\M_1)$ and $\ISOcPnp(\MT_2)$; 
    
\item[\upshape (2)]
$b(\X) = b_0(\X) = n_{\CX}(\X)$, for all
$\X$ in $\CX \coloneqq  \ISOP(\M_2)$.
\end{newlist}
\end{theorem}

\begin{proof}  (Outline)
We use, without further detail,
 the restricted Priestley duality for Ockham algebras under which each finite Ockham algebra corresponds to a finite Ockham space, that is, a finite ordered set equipped with an order-reversing self-map~$g$. For all odd $m\in \mathbb N$, let $\mathbb D_m$ be the Ockham space shown in Figure~\ref{fig:Dm} and let $\mathbf S_m$ be the corresponding Ockham algebra. Define $\mathcal M \coloneqq \{\, \mathbf S_m \mid m \text{ odd}\,\}$.

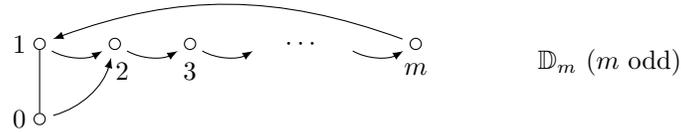
\begin{figure}[h!t]
\begin{tikzpicture}
    \begin{scope}[yshift=-3.125cm]
     \node[anchor=west] at (6.5,0.75) {$\mathbb D_m$ ($m$ odd)};
     \node[unshaded] (0) at (0,0) {};
         \node[label, anchor=east] at (0) {$0$};
     \node[unshaded] (m) at (0,1) {};
         \node[label, anchor=east] at (m) {$1$};
     \node[unshaded] (1) at (1,1) {};
         \node[label, anchor=north] at ($(1) + (0.1,0)$) {$2$};
     \node[unshaded] (2) at (2,1) {};
         \node[label, anchor=north] at (2) {$3$};
     \node[invisible] (3) at (3,1) {};
     \node at (3.5,1) {$\cdots$};
     \node[invisible] (4) at (4,1) {};
     \node[unshaded] (5) at (5,1) {};
         \node[label, anchor=north] at (5) {$m$};
     \draw[order] (0) to (m);
     \draw[curvy] (m) to [bend right] (1);
     \draw[curvy] (0) to [bend right] (1);
     \draw[curvy] (1) to [bend right] (2);
     \draw[curvy] (2) to [bend right] (3);
     \draw[curvy] (4) to [bend right] (5);
     \draw[curvy, bend angle=20] (5) to [bend right] (m);
     \end{scope}
\end{tikzpicture}
\caption{The Ockham space $\mathbb D_m$}
\label{fig:Dm}
\end{figure}

Consider a fixed $\M_1$ in our chosen family $\mathcal M$.
We take the alter ego $\M_2 = \langle M; u, \preccurlyeq, \Tp\rangle $ supplied by \cite[Theorem~3.7]{DP87}
in the simplified form described in~\cite[Theorem~5.4]{DNP}.
By~\cite[Theorem~5.7]{DNP}, if $\M_1 = \mathbf S_m$, then a structure $\X =\langle X; u, \preccurlyeq\rangle$ belongs to $\ISOcPnp(\MT_2)$ if and only if \begin{newlist}
\item[\upshape(i)] 
 $\langle X; \preccurlyeq, \Tp\rangle$ is a Priestley space,
\item[\upshape(ii)] 
$\X$ satisfies $x \preccurlyeq y \implies u(x) = u(y)$, and
\item[\upshape(iii)] 
$\X$ satisfies $x \preccurlyeq u^m(x)$.

\end{newlist} 
Since (ii) says that $u$ is both order-preserving and order-reversing and (iii) is an atomic formula, it follows immediately 
from Lemma~\ref{lem:Cornish} that (2) holds.
\end{proof}

\begin{remark} \label{rem:ock}
We may also ask whether  coincidence occurs for the classes $\ISP(\M_1)$,
with $\M_1 $ in our chosen family $\mathcal M$. Each $\M_1 \in \mathcal M$
has the property that $\CA\coloneqq \ISP(\M_1)$ is a variety (see \cite[Example~4.6]{DHP12}) or \cite[p.~183]{DP87}). As $\CA$ satisfies FDSC, the corresponding topological prevariety $\IScPnp(\MT_1)$ is standard by~\cite[Example~5.9]{CDFJ}---see the discussion following Theorem~\ref{thm:NatIsBohr} above.
Consequently, $b_0(\A) = n_{\CA}(\A)$, for all $\A\in \CA$, by Theorem~\ref{thm:NatIsBohr}.

We note that the first element, $\mathbf S_1$, in our sequence of varieties
is the much-studied variety of Stone algebras.
\end{remark}



\begin{thebibliography}{99}

\bibitem{Jo90} 
Ad\'amek, J., Herrlich, H., Strecker, G.E.:
Abstract and Concrete Categories: The Joy of Cats.
John Wiley and Sons, New York (1990);
Republished in: Reprints in Theory and Applications of Categories, no. 17 (2006), pp. 1--507. Available at \url{http://www.tac.mta.ca/tac/reprints/articles/17/tr17.pdf}


\bibitem{Ban76}  
Banaschewski, B.: Remarks on dual adjointness. In: Nordwestdeutsches Kategorienseminar,
Tagung, Bremen, 1976. Math.-Arbeitspapiere \textbf{7}, Teil A: Math. Forschungspapiere,
pp.~3--10., Univ. Bremen, Bremen (1976)

\bibitem{BCDP}  
Begum, S.\,N., Clark, D.\,M., Davey, B.\,A., Perkal, N.: Axiomatisation modulo Priestley, (preprint)


\bibitem{Ber}  
Bergman, G.\,M.:
An invitation to general algebra and universal constructions.
Henry Helson, Berkeley, CA (1998). Available at \url{http://math.berkeley.edu/~gbergman/245}

\bibitem{BGMM}  
Bezhanishvili, G., Gehrke, M., Mines, R., Morandi, P.\,J.:
Profinite completions and canonical extensions of Heyting algebras.
Order \textbf{23},
143--161 (2006)

\bibitem{BMM02} 
Bezhanishvili, G., Mines, R,, Morandi, P.\,J.:
The Priestley separation axiom for scattered spaces.
Order \textbf{19},  1--10  (2002)

\bibitem{BeMo} 
Bezhanishvili, G., Morandi, P.\,J.:
Priestley rings and Priestley order-compactifications.
Order \textbf{28}, 399--413 (2011)

\bibitem{NDftWA} 
Clark, D.\,M., Davey, B.\,A.:
Natural Dualities for the Working Algebraist.
Cambridge University Press, Cambridge (1998) 

\bibitem{CDFJ} 
 Clark, D.\,M.,  Davey, B.\,A.,  Freese, R.\,S., Jackson, M.:
 Standard topological algebras: syntactic and principal congruences and profiniteness.
 Algebra Universalis  \textbf{52}, 343--376 (2004)

\bibitem{CDHPT} 
Clark, D.\,M., Davey, B.\,A.,  Haviar, M.,  Pitkethly, J.\,G., Talukder, M.\,R.:
Standard topological quasi-varieties.
Houston J. Math.
\textbf{29}, 343--376 (2003)

\bibitem{CDMM} 
Clark, D.\,M., Davey, B.\,A., Jackson, M., Mar{\'o}ti, M., McKenzie, R.\,N.:
Principal and syntactic congruences in congruence-distributive and congruence-permutable varieties.
J. Aust. Math. Soc. \textbf{85}, 59--74 (2008)

\bibitem{CDJP} 
Clark, D.\,M., Davey, B.\,A., Jackson, M., Pitkethly, J.\,G.:
The axiomatizability of topological prevarieties.
Adv. Math.
\textbf{218}, 1604--1653 (2008)


\bibitem{CDPR} 
Clark, D.\,M., Davey, B.\,A., Pitkethly, J.\,G., Rifqui, D.\,L.: Flat unars: the primal, the semi-primal and the dualisable, Algebra Universalis \textbf{63}, 303--329 (2010)

\bibitem{CISSW} 
Clark, D.\,M., Idziak, P.\,M., Sabourin, L.\,R.,  Szab{\'o}, C.,  Willard, R.:
Natural dualities for quasivarieties generated by a finite commutative ring.
Algebra Universalis \textbf{46}, 285--320 (2001)

\bibitem{D78}  
Davey, B.\,A.:
Topological duality for prevarieties of universal algebras, Studies in Foundations and Combinatorics (
G.-C. Rota, ed.), Adv. in Math. Suppl. Stud. 1, Academic Press, New York, 1978, pp. 61--99 

\bibitem{D06} Davey, B.\,A.: 
Natural dualities for structures,
Acta Univ. M. Belii Ser. Math.
\textbf{13},
3--28  (2006). Available at \url{http://actamath.savbb.sk/pdf/acta1301.pdf}

\bibitem{nisp} 
Davey, B.\,A., Gouveia, M.\,J., Haviar, M., Priestley, H.\,A.:
Natural extensions and profinite completions of algebras.
Algebra Universalis \textbf{66}, 205--241  (2011)

\bibitem{DHP07} 
Davey, B.\,A., Haviar, M., Priestley, H.\,A.:
Boolean topological distributive lattices and canonical extensions.
Appl. Categ. Structures \textbf{15}, 225--241 (2007)

\bibitem{DHP12} 
Davey, B.\,A., Haviar, M., Priestley, H.\,A.:  Natural dualities in partnership.
Appl. Categ. Structures
\textbf{20}, 583--602 (2012)

\bibitem{pig} 
Davey, B.\,A., Haviar, M., Priestley, H.\,A.:
Piggyback dualities revisited.  Algebra Universalis (to appear) 

\bibitem{DJPT} 
Davey, B.\,A., Jackson, M., Pitkethly, J.\,G., Talukder, M.\,R.:
Natural dualities for semilattice-based algebras.
Algebra Universalis
\textbf{57}, 463--490 (2007)

\bibitem{DNP} 
 Davey, B.\,A., Nguyen, L., Pitkethly, J.\,G.: Counting relations on Ockham algebras, Algebra Universalis (to appear). Available at \url{arXiv:1501.02404}

\bibitem{DP86} 
 Davey,   B.\,A.,  Priestley,    H.\,A.:  Lattices of homomorphisms, J. Austral. Math. Soc. Ser. A \textbf{40}, 364--406 (1986)

\bibitem{DP87} 
Davey, B.\,A.,  Priestley, H.\,A.: Generalized piggyback
dualities and applications to Ockham algebras. Houston J.
Math. \textbf {13}, 151--197  (1987)



 \bibitem{DP12} 
Davey, B.\,A., Priestley, H.\,A.:
Canonical extensions and discrete dualities for finitely generated varieties of lattice-based algebras.
\emph{Studia Logica} \textbf{100},
137--161 (2012)

\bibitem{DT} 
Davey, B.\,A., Talukder, M.\,R.:
Dual categories for endodualisable Heyting algebras: optimization and axiomatization.
Algebra Universalis
\textbf{53},  331--355 (2005)

\bibitem{DFH10} 
Dikranjan, D.,  Ferrer, M.\,V., Hern{\'a}ndez, S.:
Dualities in topological groups.
Sci. Math. Jpn.
\textbf{72},
197--235 (2010)

\bibitem{Eng} 
Engelking, R.:
General topology,
 PWN---Polish Scientific Publishers, Warsaw
 (1977)

\bibitem{comp2} 
Gierz, G.,  Hofmann, K.\,H., Keimel, K., Lawson, J.\,D., Mislove, M., Scott, D.\,S.:
Continuous Lattices and Domains.
Cambridge University Press
(2003)

\bibitem{Gold81} 
Goldberg, M.S.: Distributive Ockham algebras: free algebras and injectivity. Bull. Austral. Math. Soc. \textbf{24}, 161--203 (1981)


\bibitem{GP-s1} 
Gouveia, M.\,J., Priestley, H.\,A.:
Profinite completions of semilattices and canonical extensions of semilattices and lattices.
Order \textbf{31},
189--216 (2014)

\bibitem{HK99} 
Hart, J.\,E., Kunen, K.:
Bohr compactifications of discrete structures.
Fund. Math.,
\textbf{160}, 101--151 (1999)

\bibitem{HMS74} 
Hofmann, K.\,H., Mislove, M., Stralka, A.:
The Pontryagin duality of compact ${\rm O}$-dimensional semilattices and its applications.
Lecture Notes in Mathematics \textbf{396},
Springer-Verlag  (1974)

\bibitem{Ho64} 
Holm, P.:
On the Bohr compactification.
Math.  Annalen
\textbf{156}, 34--46 (1964)

\bibitem{J08} 
Jackson, M.:  Residual bounds for compact totally
disconnected algebras.
Houston J. Math. \textbf{34},  33--67  (2008)

\bibitem{J15} 
Jackson, M.: Natural dualities, nilpotence and projective planes. Algebra Universalis 
(to appear) 


\bibitem{jez} 
Je\v{z}ek, J.: Subdirectly irreducible semilattices with an
automorphism, Semigroup Forum \textbf{43}, 178--186 (1991)

\bibitem{Jo80} 
Johnstone, P.\,T.: Stone Spaces.  Cambridge University Press (1980)

\bibitem{J10} 
Johansen, S.\,M.:
Natural dualities for three classes of relational structures.
Algebra Universalis
\textbf{63}, 149--170 (2010)

\bibitem{HBA} 
Koppelberg, S.:
Handbook of {B}oolean algebras. {V}ol. 1.
Edited by J. Donald Monk and Robert Bonnet,
North-Holland Publishing Co., Amsterdam
(1989)

\bibitem{N76} 
Nachbin, L.: Topology and order. Robert E. Krieger Publishing Co., Huntington, N. Y. (1976)

\bibitem{Nai} 
Nailana, K.\,R.:
(Strongly) zero-dimensional partially ordered spaces.
Papers in honour of Bernhard Banaschewski (Cape Town, 1996),
pp. 445--456. Kluwer, Dordrecht (2000),
reprint Springer (2010)


\bibitem{Num57} 
Numakura, K.:
Theorems on compact totally disconnected semigroups and lattices.
Proc. Amer. Math. Soc.
\textbf{8}, 623--626 (1957)

\bibitem{PD05}  
Pitkethly, J.\,G., Davey, B.\,A.:
Dualisability: Unary Algebras and Beyond.
Advances in Mathematics
\textbf{9}, Springer (2005)

\bibitem{Pont} 
Pontryagin, L.\,S.: Topological groups, 2nd ed. Gordon \& Breach, New York (1966)


\bibitem{QS} 
Quackenbush, R., Szab\'o, Cs.: Strong duality for metacyclic groups. J. Aust. Math. Soc. \textbf{73}, 377--392 (2002)

\bibitem{Sem} 
Semadeni, Z.:
Banach spaces of continuous functions.
PWN---Polish Scientific Publishers, Warsaw
(1971)

\bibitem{Str80} 
Stralka, A.:
A partially ordered space which is not a Priestley space.
Semigroup Forum
\textbf{20}, 293--297, (1980)


\end{thebibliography}
\end{document}